\documentclass[12pt,oneside,a4paper]{amsart}

\usepackage[utf8]{inputenc}
\usepackage[english]{babel}

\usepackage[a4paper,width=150mm,top=25mm,bottom=25mm]{geometry}

\usepackage{datetime} 

\usepackage[shortlabels]{enumitem}

\usepackage{latexsym}
\usepackage{amsmath,amssymb,amsthm,amsfonts}

\usepackage[colorlinks]{hyperref}	

\usepackage{tikz-cd}

\newtheorem{mainthm}{Theorem}

\newtheorem{thm}{Theorem}[section]
\newtheorem{lem}[thm]{Lemma}
\newtheorem{prop}[thm]{Proposition}

\theoremstyle{definition}
\newtheorem{defn}[thm]{Definition}
\newtheorem{example}[thm]{Example}

\theoremstyle{remark}
\newtheorem{remark}[thm]{Remark}

\numberwithin{equation}{section}

\newcommand{\C}{\mathbb{C}} 
\newcommand{\R}{\mathbb{R}} 
\newcommand{\Q}{\mathbb{Q}} 
\newcommand{\Z}{\mathbb{Z}} 
\newcommand{\A}{\mathbb{A}} 
\newcommand{\G}{\mathbb{G}} 

\renewcommand{\O}{\mathcal{O}}		

\DeclareMathOperator{\Gal}{Gal}		
\DeclareMathOperator{\Isom}{Isom}	
\DeclareMathOperator{\Aut}{Aut}		
\DeclareMathOperator{\Hom}{Hom}		
\DeclareMathOperator{\End}{End}		

\DeclareMathOperator{\image}{im}		

\DeclareMathOperator{\ab}{ab}		

\newcommand{\rec}{r}					
\DeclareMathOperator{\id}{id}		

\DeclareMathOperator{\Tr}{Tr}		
\DeclareMathOperator{\im}{Im}		
\DeclareMathOperator{\art}{art}		
\DeclareMathOperator{\res}{res}		
\DeclareMathOperator{\cyc}{cyc}		
\DeclareMathOperator{\Cl}{Cl}		
\DeclareMathOperator{\Sh}{Sh}		

\DeclareMathOperator{\mo}{mod}		

\DeclareMathOperator{\sgn}{sgn}		

\DeclareMathOperator{\GL}{GL}		
\DeclareMathOperator{\SL}{SL}		
\DeclareMathOperator{\ad}{ad}		
\DeclareMathOperator{\der}{der}		

\DeclareMathOperator{\Set}{Set}		

\DeclareMathOperator{\CM}{CM} 		
\DeclareMathOperator{\VZ}{\mathfrak{S}} 		

\newcommand{\plecgpsemi}{S_\Sigma\ltimes\Gamma_F^\Sigma} 
\newcommand{\plecgpintro}{S_r\ltimes\Gamma_F^r} 			
\newcommand{\plecgpaut}{\Aut_F(F\otimes_\Q\overline{\Q})} 
\newcommand{\plecgphash}{\Gamma_\Q\#\Gamma_F} 			
\DeclareMathOperator{\pl}{pl}							
\newcommand{\plecgp}{\Gamma^{\pl}} 						
\newcommand{\plecgpcm}{\Gamma^{\pl, R}_{\CM}}				
\newcommand{\plecgppi}{\Gamma^{\pl, R}_{\pi_0}}				

\DeclareMathOperator{\product}{prod}			
\newcommand{\Product}{P}						
\newcommand{\plecpiprod}{\Product_{\pi_0}}			

\begin{document}

\title[Plectic Galois action on Hilbert modular varieties]{Plectic Galois action on CM points and connected components of Hilbert modular varieties}
\author{Marius Leonhardt}
\address{Mathematisches Institut, Universität Heidelberg, Im Neuenheimer Feld 205, 69120 Heidelberg, Germany}
\email{mleonhardt@mathi.uni-heidelberg.de}

\begin{abstract}
	We expand on Nekov\'a\v{r}'s construction of the plectic half transfer to define a plectic Galois action on Hilbert modular varieties.
	More precisely, we study in a unifying fashion Shimura varieties associated to groups that differ only in the centre from $R_{F/\Q}\GL_2$.
	We define plectic Galois actions on the CM points and on the set of connected components of these Shimura varieties, and show that these two actions are compatible.
	This extends the plectic conjecture of Nekov\'a\v{r}--Scholl.
\end{abstract}

\maketitle

\tableofcontents

\section{Introduction}

In this article we study new symmetries in the theory of complex multiplication (CM).
The Main Theorem of CM due to Tate \cite{tate} and Deligne \cite{ln900IV} describes the Galois conjugate of a CM abelian variety and consequently the action of the absolute Galois group $\Gamma_\Q$ of $\Q$ on the CM points of certain Shimura varieties.
For one of these Shimura varieties --- the PEL Hilbert modular variety --- Nekov\'a\v{r} \cite{ne} extended this Galois action on the CM points to a certain subgroup of the plectic Galois group defined below.

The main purpose of this article is to extend Nekov\'a\v{r}'s results to all Shimura varieties whose associated groups differ only in the centre from $R_{F/\Q}\GL_2$, where $F$ is a totally real field.
This is motivated by the so-called ``plectic conjecture'' of Nekov\'a\v{r}--Scholl \cite{nescholl}.
It predicts an additional ``plectic'' structure of motivic origin on the Shimura variety associated to $R_{F/\Q} \GL_2$ --- the (non-PEL) Hilbert modular variety.
By studying all groups that differ in the centre from $R_{F/\Q}\GL_2$, we both generalise and interpolate between \cite{ne} and \cite{nescholl}.

Fix a totally real field $F$ of degree $[F:\Q]=r$.
The plectic Galois group is the semi-direct product $\plecgpintro$ of the symmetric group $S_r$ with $r$-tuples of elements of the absolute Galois group of $F$.
A choice of coset representatives $s_i$ for $\Gamma_\Q/\Gamma_F$ yields an embedding of $\Gamma_\Q$ into $\plecgpintro$ by sending $\gamma$ to $(\sigma,(h_i)_{1\leq i\leq r})$ determined by
\[
\gamma s_i = s_{\sigma(i)} h_i, \quad 1\leq i\leq r.
\]

The PEL Hilbert modular variety $\Sh(G_0,X_0)$ is the Shimura variety associated to the group $G_0$ defined by the Cartesian diagram
\begin{equation}\label{diag:defG-HMV}
\begin{tikzcd}
G_0 \arrow[r, hook] \arrow[d]  & R_{F/\mathbb{Q}}\GL_{2} \arrow[d, "R_{F/\mathbb{Q}}(\det)"] \\
\mathbb{G}_m \arrow[r, hook] & R_{F/\mathbb{Q}}\mathbb{G}_m.
\end{tikzcd}
\end{equation}

The Shimura variety $\Sh(G_0,X_0)$ is of PEL type; it parametrises isomorphism classes of abelian varieties with real multiplication by $F$ equipped with extra structure.
If $P$ is a special point of $\Sh(G_0,X_0)$, then $P$ corresponds to a CM abelian variety $A$ of type $(K,\Phi)$ for a totally imaginary quadratic extension $K$ of $F$ and a CM type $\Phi$ of $K$.
For $\gamma\in\Gamma_\Q$, the Galois conjugate ${}^\gamma\! P$ of $P$ thus corresponds to the Galois conjugate ${}^\gamma\! A$ of $A$.
The Main Theorem of CM describes ${}^\gamma\! A$ via Tate's half transfer $F_\Phi(\gamma)\in\Gamma_K^{\ab}$ and a suitably normalised preimage of $F_\Phi(\gamma)$ under the Artin map, the Taniyama element
\[
f_\Phi(\gamma)\in \A_{K,f}^\times/K^\times.
\]

To extend the Galois action on the CM points of $\Sh(G_0,X_0)$ to a plectic Galois action, Nekov\'a\v{r} defines a plectic half transfer $F_\Phi\colon \plecgpintro \to \Gamma_K^{\ab}$ extending Tate's half transfer.
For elements of a certain subgroup $(\plecgpintro)_0$ of $\plecgpintro$, he defines a plectic Taniyama element taking values in $\A_{K,f}^\times/K^\times$ as a class-field-theoretic preimage of the plectic half transfer.
The action of $(\plecgpintro)_0$ on the CM points of $\Sh(G_0,X_0)$ is then defined by replacing the Taniyama element by its plectic counterpart.

The group $G_0$ is closely related to $R_{F/\Q}\GL_2$. 
In this article we study Shimura varieties associated to groups $G$ that are related to $R_{F/\Q}\GL_2$ by the Cartesian diagram
\begin{equation}\label{diag:defG-HMV2}
\begin{tikzcd}
G \arrow[r, hook] \arrow[d]  & R_{F/\mathbb{Q}}\GL_{2} \arrow[d, "R_{F/\mathbb{Q}}(\det)"] \\
R \arrow[r, hook] & R_{F/\mathbb{Q}}\mathbb{G}_m,
\end{tikzcd}
\end{equation}
that is, we replace $\G_m$ in \eqref{diag:defG-HMV} with an arbitrary $\Q$-algebraic torus $R$ with $\G_m \subset R \subset R_{F/\Q}\G_m$.
For example, for $R=R_{F/\Q}\G_m$ the group $G$ is equal to $R_{F/\Q}\GL_2$ whose associated Shimura variety is the non-PEL Hilbert modular variety.

For the rest of this introduction, fix such a torus $R$ and the corresponding group $G$.
In Definition \ref{def:plec-CM} we define a subgroup $(\plecgpintro)^R_{\CM}$ of $\plecgpintro$.
The first result is

\begin{mainthm}[Theorem \ref{thm:plecact-CM}]\label{thm:plecact-CM-intro}
	The group $(\plecgpintro)^R_{\CM}$ acts on the CM points of the Shimura variety $\Sh(G,X)$, extending the action of $\Gamma_\Q$.
\end{mainthm}

To prove it, we follow a strategy similar to Nekov\'a\v{r}'s.
Namely, we start by describing the points of $\Sh(G,X)$ in terms of abelian varieties with real multiplication equipped with extra structure, see Theorem \ref{thm:moduliHMV-general}, and use Nekov\'a\v{r}'s plectic half transfer.
The main new contribution of this article is to define the plectic Taniyama element on the entire plectic group $\plecgpintro$, see Definition \ref{def:plecTanielt}.
The definition depends on a choice of splitting $\chi_F$ of the Artin map $\rec_F\colon \A^\times_{F,f}/F^\times_{>0} \to \Gamma_F^{\ab}$.

The definition of the plectic Taniyama element on all of $\plecgpintro$ provides enough flexibility to define a plectic action on CM points for every choice of $R$.
In the theorem, we still have to restrict to a subgroup $(\plecgpintro)^R_{\CM}$ of $ \plecgpintro$ because otherwise the action on the polarisation class of the CM abelian variety would not be well defined.
For example, in the special case where $R=R_{F/\Q} \G_m$, the subgroup $(\plecgpintro)^R_{\CM}$ is equal to $\plecgpintro$.

However, to the best of our knowledge the plectic action on CM points does depend on the choice of splitting $\chi_F$.
In particular, in the case of $R=R_{F/\Q} \G_m$ where the action of the full plectic group should (conjecturally) be canonical this dependence remains a mystery.

The second main contribution of this article is the definition of a canonical plectic action on the connected components of the Shimura varieties $\Sh(G,X)$, independent of the choice of $\chi_F$.
Namely, in Definition \ref{def:pi0-plec} we define another plectic group $(\plecgpintro)^R_{\pi_0}$ and an action of this group on the set $\pi_0(\Sh(G,X))$ of connected components of the Shimura variety $\Sh(G,X)$.
This action extends the $\Gamma_\Q$-action, and moreover the group $(\plecgpintro)^R_{\CM}$ canonically embeds into $(\plecgpintro)^R_{\pi_0}$, see Proposition \ref{lem:plecCM-inside-pi0}.
Using the description of the set of connected components of a Shimura variety as a zero-dimensional Shimura variety in combination with the plectic actions defined above, we prove our main result:

\begin{mainthm}[Theorem \ref{thm:pi0-plec-equiv}]\label{thm:pi0-plec-equiv-intro}
	The $\pi_0$-map restricted to CM points is $(\plecgpintro)^R_{\CM}$-equivariant.
\end{mainthm}

We view this result as a sanity check for the soundness of the plectic conjectures.
Moreover, as the plectic action on connected components is independent of $\chi_F$, this theorem reassures us that the dependence of the plectic action on CM points on the choice of $\chi_F$ is relatively mild. 


Nekov\'a\v{r}--Scholl conjecture a plectic structure for Shimura varieties associated to groups of the form $R_{F/\Q} H$, for a reductive group $H$ over $F$.
In \cite{nescholl} they outline how this plectic structure should manifest in various realisations and sketch arithmetic applications to special values of $L$-functions.
For example, on the \'etale cohomology groups of Shimura varieties, a plectic structure is simply an action of the plectic group extending the Galois action.

In this article, we describe plectic structures on CM points and on the set of connected components for the Shimura variety associated to $G$.
However, the group $G$ is not of the form $R_{F/\Q} H$ for any $H$ unless $R=R_{F/\Q}\G_m$, in which case $G=R_{F/\Q}\GL_2$.
In this special case, our results fit into the framework of \cite{nescholl}. 
Moreover, we show that if we want the \emph{full} plectic group to act (on either CM points or on connected components), it is essential to work with the Shimura variety with the ``largest possible centre'' $R=R_{F/\Q}\G_m$.
And it is the full plectic group that is conjecturally responsible for arithmetic applications.

By proving our results for all $R$ and hence for all $G$ that differ in the centre from $R_{F/\Q}\GL_2$, we add some flexibility to the plectic framework.
In the case $R=\G_m$ the action in Theorem \ref{thm:plecact-CM-intro} is precisely the one discovered in \cite{ne}; we thus bridge the gap between \cite{ne} and \cite{nescholl}.
The definition of the plectic action on $\pi_0(\Sh(G,X))$ and Theorem \ref{thm:pi0-plec-equiv-intro} are new results even for $R\in\{\G_m, R_{F/\Q}\G_m\}$.

For other plectic structures in an Archimedean setup, see \cite{neschollHodge} for the Hodge realisation and \cite{xi-plecgreen} for a plectic Green function and applications to multiple zeta values.
Moreover, \cite{blake} constructs a plectic Taniyama group similar to Langlands' Taniyama group \cite{ln900III}.
Tamiozzo \cite{tamiozzothesis} proved a function field analogue of the plectic conjecture and proposed a local variant of the conjectures, using Scholze's diamonds and mixed characteristic shtukas.

\subsection*{Structure of the article}

Section \ref{se:setup} serves a threefold purpose.
We introduce notation, study the Shimura varieties $\Sh(G,X)$, and present the results of CM theory in a form most amenable to generalisation.
In \S\ref{sse:hmv} we prove that $G$ and $R_{F/\Q}\GL_2$ differ only in the centre, see Lemma \ref{lem:Gad}, and that the complex points of $\Sh(G,X)$ can be described using abelian varieties equipped with real multiplication and an $R(\Q)$-class of a polarisation, see Theorem \ref{thm:moduliHMV-general}.
In particular, this allows us to study the special, i.\,e.\ CM, points of $\Sh(G,X)$ in \S\ref{sse:CMpts}.
For a fixed CM field $K$, we describe the $\Gamma_\Q$-action on points of $\Sh(G,X)$ with CM by $K$ of CM type $\Phi$ using Tate's half transfer $F_\Phi\colon \Gamma_\Q \to \Gamma_K^{\ab}$, the Taniyama element
\[
f_\Phi\colon \Gamma_\Q \longrightarrow \A_{K,f}^\times/K^\times,
\]
and the Main Theorem of CM, see Theorem \ref{thm:conjCMpt}.

In Section \ref{se:plecCM} we explain the plectic generalisation of CM theory.
We start \S\ref{sse:plecgp} with properties of the plectic Galois group $\plecgp$, a choice-free version of $\plecgpintro$, and then recall the definition of Nekov\'a\v{r}'s plectic half transfer $F_\Phi\colon \plecgp \longrightarrow \Gamma_K^{\ab}$.
In \S\ref{sse:plecTani} we use a splitting $\chi_F$ of the reciprocity homomorphism $\rec_F\colon \A_{F,f}^\times/F^\times_{>0} \to \Gamma_F^{\ab}$ to define the plectic Taniyama element
\[
f_\Phi\colon \plecgp \longrightarrow \A_{K,f}^\times/K^\times
\]
as a suitably normalised preimage of $F_\Phi$ under $\rec_K\colon \A_{K,f}^\times/K^\times \to \Gamma_K^{\ab}$.
We prove that the plectic Taniyama element extends the (non-plectic) Taniyama element, see Lemma \ref{lem:sameTanielt}.
Then we define the group $\plecgpcm$ and an action of this group on the CM points of $\Sh(G,X)$, proving Theorem \ref{thm:plecact-CM}.

Section \ref{se:plecpi0} is devoted to the set $\pi_0(\Sh(G,X))$ of connected components of the Shimura variety $\Sh(G,X)$.
Using the general machinery of Shimura varieties, we prove in Lemma \ref{lem:pi0SV} that this set is equal to the zero-dimensional Shimura variety $\Sh(R,\VZ)$ and recall the action of $\Gamma_\Q$ on it, which is given by a certain reciprocity homomorphism, see Definition \ref{def:Galact-pi0}.
We show in Lemma \ref{lem:pi0CR} that $\pi_0(\Sh(G,X))=\pi_0(R(\A)/R(\Q))$, enabling us to define the plectic group $\plecgppi$ and an action of this group on $\pi_0(\Sh(G,X))$ extending the Galois action, see Definition \ref{def:pi0-plec}.
In Proposition \ref{lem:plecCM-inside-pi0} we prove that $\plecgpcm$ canonically embeds into $\plecgppi$, and then prove our main result in Theorem \ref{thm:pi0-plec-equiv}.

\subsubsection*{Acknowledgements}

This article is a condensed version of my PhD thesis \cite{leon-thesis}, carried out at the University of Cambridge.
I would like to thank my supervisor Tony Scholl for his ongoing support, and everyone who proofread parts of my thesis.
I would also like to thank Fred Diamond and Shu Sasaki for very helpful discussions about Hilbert modular varieties, and the anonymous referee for various helpful suggestions to improve the exposition.
During the PhD I was supported by Trinity College and EPSRC studentship \#1648608.
I also acknowledge support from the Deutsche Forschungsgemeinschaft (DFG, German Research Foundation) through TRR 326 Geometry and Arithmetic of Uniformized Structures, project number 444845124.

\section{Set-up}\label{se:setup}

\subsection{Notation}\label{sse:not}

We denote the adeles (resp.\ finite adeles) of a number field $k$ by $\A_k$ (resp.\ $\A_{k,f}$); the ideles (resp.\ finite ideles) of $k$ are denoted by $\A_k^\times$ (resp.\ $\A_{k,f}^\times$).
If $k=\Q$, we usually drop the index $\Q$ from notation.
The idele class group is denoted by $C_k=\A_k^\times/k^\times$.

Throughout, we let $\overline{\Q}$ be the algebraic closure of $\Q$ inside $\C$ and denote complex conjugation by $c$ or $z \mapsto \bar{z}$.
For simplicity of notation, we assume that all number fields are embedded into $\overline{\Q}$.
For a number field $k$, we let $\Gamma_k:=\Gal(\overline{\Q}/k)$ be the absolute Galois group of $k$.
We write $\chi_{\cyc}\colon \Gamma_{\Q} \to \hat{\Z}^\times$ for the cyclotomic character.

We write $\art_k\colon C_k \to \Gamma_k^{\ab}$ for the Artin map of $k$, a topological group homomorphism characterised by sending uniformisers to geometric Frobenius elements; the Artin map is surjective with kernel equal to the identity component $C_k^0$ of $C_k$ and thus induces an isomorphism $\art_k\colon \pi_0(C_k) \xrightarrow{~\sim~} \Gamma_k^{\ab}$.
The Artin map also induces a surjective homomorphism $\rec_k\colon \A_{k,f}^\times/k^\times_{>0} \twoheadrightarrow \Gamma_k^{\ab}$, where $k^\times_{>0}$ denotes the totally positive elements of $k$.

If $k'/k$ is a finite extension of number fields, then $\Gamma_{k'}$ is a finite index subgroup of $\Gamma_k$ and we get the transfer map $V_{k'/k}\colon \Gamma_k^{\ab} \to \Gamma_{k'}^{\ab}$.
Its compatibility with the Artin maps of $k$ and $k'$ is expressed by the commutative diagram \cite[(13), p.\ 197]{casfro-gcft}
\begin{equation*}
\begin{tikzcd}
\A_{k'}^\times/k'^\times \arrow[r, "\art_{k'}"]     & \Gamma_{k'}^{\ab}                        \\
\A_k^\times/k^\times \arrow[r, "\art_k"'] \arrow[u] & {\Gamma_k^{\ab},} \arrow[u, "V_{k'/k}"']
\end{tikzcd}
\end{equation*}
and the analogous diagram for the maps $r_k$ and $r_{k'}$ commutes, too.

\subsection{Variants of the Hilbert modular variety}\label{sse:hmv}

Throughout this article, fix a totally real field $F\subset \overline{\Q}$.
Let $\Sigma:=\Sigma_F:=\Hom(F,\overline{\Q})$ be the set of embeddings of $F$ into $\overline{\Q}$.
Let $G_1$ be the algebraic group $R_{F/\Q}\GL_2$ (Weil restriction of scalars) defined over $\Q$, and let $X_1$ be the $G_1(\R)$-conjugacy class of the morphism $h\colon \mathbb{S} \to (G_1)_\R$ that is given, on $\R$-points, by
\begin{align*}
h(i)=\left(\begin{pmatrix} 0 & -1 \\ 1 & 0 \end{pmatrix}\right)_{x\in\Sigma_F} \in \GL_2(\R)^\Sigma.
\end{align*}
Here $\mathbb{S}$ denotes the Deligne torus $R_{\C/\R}\G_m$; also note that $G_1(\R)=\GL_2(\R)^\Sigma$.
By letting $\GL_2(\R)^\Sigma$ act on $(\C\setminus\R)^\Sigma$ by componentwise M\"obius transformations, we identify $X_1$ with $(\C\setminus\R)^\Sigma$ by mapping $h$ to $(i,\dots,i)$.
The pair $(G_1,X_1)$ is a Shimura datum, i.\,e.\ it satisfies the axioms \cite[(2.1.1.1-3)]{del-vds}.
The associated Shimura variety
\[
\Sh(G_1,X_1):=\varprojlim_U G_1(\Q)\backslash \left[X_1 \times G_1(\A_f)/U\right] = G_1(\Q)\backslash\left[X_1\times G_1(\A_f)/\overline{Z_1(\Q)}\right]
\]
is called the \emph{Hilbert modular variety} associated to $F$.
Here $U$ runs over all compact open subgroups of $G_1(\A_f)$, $Z_1$ denotes the centre of $G_1$, and the second identity is \cite[(1)]{orr}, see also \cite[Prop.\ 2.1.10]{del-vds}.
Note that in this article we will usually work with the projective limit as above, which is a Shimura \emph{pro-variety} in the terminology of \cite{orr}, being the projective limit of the varieties $\Sh_U(G_1,X_1)$.
We have chosen to simply call $\Sh(G_1,X_1)$ a Shimura variety as well as this should not cause any confusion.

We study variants of the Hilbert modular variety by varying the centre of $G_1$.
To that end, fix an algebraic torus $R$ over $\Q$ with
\[
\G_m \hookrightarrow R \hookrightarrow R_{F/\Q}\G_m
\]
such that the composite morphism $\G_m \hookrightarrow R_{F/\Q}\G_m$ is given by the inclusion $\Q^\times \hookrightarrow F^\times$ on $\Q$-points.
We define the algebraic groups $G_0$ and $G$ by the two Cartesian squares
\begin{equation}\label{diag:defG}
\begin{tikzcd}
G_0 \arrow[r, hook] \arrow[d] & G \arrow[r, hook] \arrow[d] & G_1 \arrow[d, "d"] \\
\G_m \arrow[r, hook]          & R \arrow[r, hook]           & R_{F/\Q}\G_m,
\end{tikzcd}
\end{equation}
compare \eqref{diag:defG-HMV} and \eqref{diag:defG-HMV2}.
On $\Q$-points, the morphism $d$ is the determinant map $\GL_2(F) \to F^\times$.
More conceptually, it fits into the short exact sequence
\begin{equation} \label{diag:ses-Gder}
\begin{tikzcd}
1 \arrow[r] & G_1^{\der} \arrow[r] & G_1 \arrow[r,"d"] & R_{F/\Q}\G_m \arrow[r] & 1,
\end{tikzcd}
\end{equation}
where $G_1^{\der}$ denotes the derived group of $G_1$.
It is equal to $R_{F/\Q}\SL_2$.

\begin{lem}\label{lem:Gad}
	$G$ and $G_1$ have the same derived group.
	Consequently, they also have the same adjoint group, the centre $Z$ of $G$ is equal to $G \cap Z_1$, and $G/G^{\der}=R$.
\end{lem}

We therefore say that $G$ \emph{differs only in the centre} from $G_1$.

\begin{proof}
	By \eqref{diag:ses-Gder} we have $\ker(d)=G_1^{\der}$.
	By \eqref{diag:defG}, the map $d|_G$ factors through $R$, and $R$ is commutative, hence $\ker(d)$ contains $G^{\der}$.
	On the other hand, the Cartesian diagram \eqref{diag:defG} also implies that $G$ contains $\ker(d)=G_1^{\der}$, and hence $G^{\der}$ contains $(G_1^{\der})^{\der} = G_1^{\der}$, where the last equality holds because $G_1^{\der}$ is semisimple.
	We conclude that $G^{\der}=G_1^{\der}$.
	
	Moreover, the adjoint group of a reductive group is the same as the adjoint group of its derived group, thus $G^{\ad}=G_1^{\ad}$.
	Then
	\[
	Z=\ker(G\to G^{\ad}) = \ker(G\hookrightarrow G_1 \to G_1^{\ad} = G^{\ad}) = G \cap Z_1.
	\]
	
	Finally,
	\[
	\ker(d|_G\colon G \to R) = \ker(d)\cap G = G_1^{\der}\cap G = G^{\der}.
	\]
\end{proof}

The morphism $h\colon \mathbb{S} \to (G_1)_\R$ factors through $(G_0)_\R$. 
Let $X$ be the $G(\R)$-conjugacy class of $h$.
We have $X\subset X_1$, and we can describe $X$ more concretely in terms of upper and lower half planes.
Let $\sgn\colon \R^\times \to \{\pm 1\}$ be the sign function, and define the subgroup 
\[
\VZ:=\left(R(\R)\cdot (\R^\times_{>0})^{\Sigma}\right)/(\R^\times_{>0})^{\Sigma} \subset (\R^\times)^{\Sigma}/(\R^\times_{>0})^{\Sigma} =\{\pm 1\}^{\Sigma}.
\]

\begin{lem}\label{lem:XVZ}
	Under the identification $X_1 = (\C\setminus\R)^\Sigma$, the $G(\R)$-conjugacy class $X$ corresponds to
	\begin{align}\label{def:X-VZ} 
	\{(z_x)_{x\in\Sigma}\in (\C\setminus\R)^{\Sigma}~|~ (\sgn \im z_x)_{x\in\Sigma} \in \VZ\}.
	\end{align}
\end{lem}

\begin{proof} 
	Let us temporarily denote the set in \eqref{def:X-VZ} by $W$.
	We need to show that the $G(\R)$-orbit of $(i,\dots,i)\in (\C\setminus\R)^\Sigma$ is equal to $W$.
	By definition of $G$ in \eqref{diag:defG} and the usual formula for the imaginary part under M\"obius transformations in terms of the determinant, we see that $G(\R)\cdot (i,\dots,i)\subset W$.
	
	Conversely, for any $(z_x)_x\in W$ it is easy to find an element $(z'_x)_x\in G(\R)\cdot (i,\dots,i)$ with $\sgn\im z_x = \sgn \im z'_x$.
	Namely, by \eqref{def:X-VZ} there is an element $g=(g_x)_{x\in\Sigma}$ of $G(\R)\subset G_1(\R)=\GL_2(\R)^\Sigma$ satisfying $\sgn \det g_x = \sgn\im z_x$ for all $x\in\Sigma$, thus we may take $(z'_x)_{x\in\Sigma} = g\cdot (i,\dots,i)$.
	Moreover, $\GL_2(\R)$ acts transitively on $\C\setminus\R$, so there exists $g'=(g'_x)_x \in \GL_2(\R)^\Sigma$ with $g'_x \cdot z'_x = z_x$ for all $x\in\Sigma$.
	By looking at signs of imaginary parts, we must have $\det g'_x > 0$ for all $x\in\Sigma$, hence we may rescale $g'$ so that $g'\in \SL_2(\R)^\Sigma \subset G(\R)$.
	We conclude that $(z_x)_x\in G(\R)\cdot (i,\dots,i)$.
\end{proof}

By Lemma \ref{lem:Gad} we have $G^{\ad}=G_1^{\ad}$.
The axioms \cite[(2.1.1.1-3)]{del-vds} of being a Shimura datum really only depend on the adjoint group.
We already saw that $(G_1,X_1)$ satisfies these axioms, so we conclude that $(G,X)$ is a Shimura datum, too.
We call the associated Shimura variety $\Sh(G,X)$ a \emph{variant} of the Hilbert modular variety.

\begin{example}(PEL Hilbert modular variety)\label{exa:PELHMV}
	In the case $R=\G_m$, we have $G=G_0$, $\VZ=\{\pm (1,\dots,1)\}$ and so $X=X_0:=\mathfrak{h}^\Sigma \sqcup (-\mathfrak{h})^\Sigma$, where $\mathfrak{h}\subset\C$ denotes the upper half plane.
	The Shimura datum $(G_0,X_0)$ is of \emph{PEL type}, see \cite[Def.\ 8.15]{misv}, \cite[4.9]{del}.
	
	To be precise, it is associated to the type (C) PEL datum consisting of the simple $\Q$-algebra $F$, with trivial involution, acting on the $\Q$-vector space $V=F^2$ equipped with the alternating, $\Q$-bilinear, $F$-compatible form $\psi\colon V \times V \to \Q$ given by
	\[
	\psi\left(\begin{pmatrix} v_1 \\ v_2\end{pmatrix},\begin{pmatrix} w_1 \\ w_2\end{pmatrix}\right)=\Tr_{F/\Q}\circ\det\begin{pmatrix} v_1 & w_1 \\ v_2 & w_2\end{pmatrix}.
	\]
	
	We call the associated Shimura variety $\Sh(G_0,X_0)$ the \emph{PEL Hilbert modular variety}.
\end{example}

The Shimura datum $(G,X)$ depends on the choice of the intermediate torus $\G_m \hookrightarrow R \hookrightarrow R_{F/\Q}\G_m$.
We think of the family of Shimura varieties $\Sh(G,X)$ as interpolating between the Hilbert modular case ($R=R_{F/\Q}\G_m$) and the PEL Hilbert modular case ($R=\G_m$).
Our goal is to develop a plectic theory for all these variants of the Hilbert modular variety in order to bridge the gap between  the results of \cite{ne} and \cite{nescholl}.
We start by relating the complex points of the Shimura variety $\Sh(G,X)$ to isomorphism classes of abelian varieties equipped with extra structure.

First we look at quadruples $(A,i,s,\eta)$.
Here $A$ denotes a complex abelian variety of dimension $[F:\Q]$ equipped with real multiplication by $F$ via the ring homomorphism $i\colon F\hookrightarrow \End(A)\otimes_\Z\Q$.
Moreover, $s$ is a polarisation of $A$, which we think of as a Riemann form $s\colon H_1(A,\Q)\times H_1(A,\Q) \to \Q$.
We require that $s$ is $F$-compatible, meaning that $(f\cdot s)(u,v):=s(i(f)u,v)$ is equal to $s(u,i(f)v)$ for all $u,v\in H_1(A,\Q)$ and $f\in F$.
Finally, the level structure $\eta$ is an $\A_{F,f}$-module-isomorphism $\eta\colon V\otimes_\Q \A_f \xrightarrow{~\sim~} \widehat{V}(A)$.
Here $(V,\psi)$ are as in Example \ref{exa:PELHMV} and $\widehat{V}(A)=\widehat{T}(A)\otimes_\Z\Q$, where $\widehat{T}(A):=\varprojlim_n A[n]$ is the full Tate module of $A$.

More precisely, we are interested in quadruples $(A,i,R(\Q)s,\eta\overline{Z(\Q)})$, i.\,e.\ not $s$ (resp.\ $\eta$) itself is part of the datum, but only the class of all $R(\Q)$-multiples (resp.\ $\overline{Z(\Q)}$-translates) of $s$ (resp.\ $\eta$).
Additionally, we require that $\eta$ sends the class $R(\A_f)\psi$ to $R(\A_f)s$ and that there exists an $F$-linear isomorphism
\begin{align}\label{eq:doublestar}
a\colon H_1(A,\Q)\xrightarrow{~\sim~} V
\end{align}
that sends $R(\Q) s$ to $R(\Q) \psi$ and satisfies $a\circ h_A \circ a^{-1}\in X$. 
Here $h_A\colon \mathbb{S} \to \End(H_1(A,\R))$ denotes the Hodge structure on $H_1(A,\Q)$.

We call two quadruples $(A,i,R(\Q)s,\eta \overline{Z(\Q)})$ and $(A',i',R(\Q)s',\eta' \overline{Z(\Q)})$ isomorphic if there exists a quasi-isogeny $f\colon A \to A'$ that is $F$-linear (with respect to $i$ and $i'$), sends $R(\Q)s$ to $R(\Q)s'$ and satisfies $\eta' \overline{Z(\Q)}=f\circ \eta \overline{Z(\Q)}$.
We denote the set of isomorphism classes of such quadruples by $\mathcal{A}(\C)$.
We employ the strategy of \cite[\S 6]{misv} to identify $\mathcal{A}(\C)$ with the complex points of the Shimura variety $\Sh(G,X)$.
Namely, we define
\[
\alpha\colon \Sh(G,X)(\C) \to \mathcal{A}(\C)
\]
by mapping $[h,g]$ to the isomorphism class of $(A_h,i_0,R(\Q)\psi,g\overline{Z(\Q)})$, where $A_h$ is the abelian variety with Hodge structure $H_1(A_h,\Q)$ equal to $(V,h)$, the ring homomorphism $i_0 \colon F \to \End(A)\otimes_\Z\Q = \End(V,h)$ is the obvious one (note that $V=F^2$), the pairing $\psi\colon V\times V \to \Q$ is as above, and $g$ is viewed as the map $g\colon V\otimes_\Q\A_f \to V\otimes_\Q\A_f = \widehat{V}(A_h)$.

Conversely, we define
\[
\beta\colon \mathcal{A}(\C) \to \Sh(G,X)(\C)
\]
by mapping the isomorphism class of $(A,i,R(\Q)s,\eta\overline{Z(\Q)})$ to $[a\circ h_A \circ a^{-1},a\circ\eta]$, where $a\colon H_1(A,\Q) \xrightarrow{~\sim~} V$ is an isomorphism as in \eqref{eq:doublestar}.
It is straightforward to check that $\alpha$ and $\beta$ are well-defined and inverse to each other:

\begin{thm}\label{thm:moduliHMV-general}
	Let $\G_m \hookrightarrow R \hookrightarrow R_{F/\Q}\G_m$ be an intermediate algebraic torus over $\Q$ and $(G,X)$ be the associated Shimura datum as above.
	Then the maps $\alpha$ and $\beta$ are mutually inverse bijections between the complex points of the Shimura variety $\Sh(G,X)$ and the set $\mathcal{A}(\C)$.
\end{thm}

\begin{remark}
	In the case $R=\G_m$, this is a special case of the description of the complex points of a PEL Shimura variety in terms of abelian varieties, see \cite[4.11]{del}.
	In the case $R=R_{F/\Q}\G_m$, Theorem \ref{thm:moduliHMV-general} is a special case of \cite[4.14]{del}. 
\end{remark}

\subsection{CM points}\label{sse:CMpts}

The reflex field of the Shimura datum $(G,X)$ is $\Q$. 
The canonical model, also denoted $\Sh(G,X)$, of the Shimura variety $\Sh(G,X)$ is a projective system of varieties $\Sh_U(G,X)$ over $\Q$ that are (for small enough compact open subgroups $U\subset G(\A_f)$) coarse moduli spaces for functors $\mathcal{A}$ modelled on the set $\mathcal{A}(\C)$ of Theorem \ref{thm:moduliHMV-general}.
To prove this, one follows the strategy of \cite[\S 2]{fred}, where the case of $\Sh_U(G_1,X_1)$ is done in detail.
We only give a sketch here for the models of $\Sh_{U(N)}(G,X)$ for the principal congruence subgroups
\[
U(N):=G(\A_f)\cap \left\{\left. g\in \GL_2(\widehat{\mathcal{O}}_F) ~\right|~g \equiv \begin{pmatrix} 1 & 0\\ 0 & 1 \end{pmatrix} \mo N\right\}
\]
with $N$ large enough.
For a non-zero fractional ideal $J$ of $F$, \cite[2.2.1]{fred} construct the fine moduli space 
$Y_{J,N}$ parametrising isomorphism classes of abelian varieties equipped with a $J$-polarisation and full level-$N$-structure.
Let $Z_{J,N}$ be the $\Q$-scheme representing $\O_F$-linear isomorphisms $J/NJ \xrightarrow{~\sim~}\mathfrak{d}^{-1}\otimes\mu_N$, where $\mathfrak{d}$ denotes the different of $F$.
There is a canonical map $Y_{J,N} \to Z_{J,N}$, which base changed to $\C$ is precisely the $\pi_0$-map.
Note that a choice of primitive $N$-th root of unity $\zeta_N$ gives an isomorphism
\begin{align}\label{iso:OFmodN}
\Isom_{\O_F}(J/NJ,\mathfrak{d}^{-1}\otimes\mu_N) \cong (\O_F/N\O_F)^\times.
\end{align}

Now define $R(\Q)_{>0} := R(\Q)\cap F^\times_{>0}$, $R(\Z)_{>0} := R(\Q)_{>0} \cap \mathcal{O}_F^\times$, $R(\widehat{\Z}) := R(\A_f)\cap\widehat{\mathcal{O}}_F^\times$, and $R(\Z/N\Z) := R(\widehat{\Z})/d(U(N))$. 
Note that $(\Z/N\Z)^\times \subset R(\Z/N\Z) = R(\A_f)\cap \widehat{\mathcal{O}}_F^\times / R(\A_f)\cap (1+N\widehat{\mathcal{O}}_F)^\times \subset (\O_F/N\O_F)^\times$.
We denote the preimage of $R(\Z/N\Z)$ under \eqref{iso:OFmodN} by $\Isom_{\O_F}^R \subset \Isom_{\O_F}(J/NJ,\mathfrak{d}^{-1}\otimes\mu_N)$; it is independent of the choice of $\zeta_N$.

For $\mathfrak{c}\in \Cl_R^+ := R(\A_f)/R(\Q)_{>0} R(\widehat{\Z})$, the ``narrow class group'' of $R$, let $[J_{\mathfrak{c}}] = [\mathfrak{d}^{-1}]\iota(\mathfrak{c})$, where $\iota\colon \Cl_R^+ \to \Cl_F^+$ denotes the natural map.
Let $Z_{\mathfrak{c},N}^R$ be the subscheme of $Z_{J_\mathfrak{c},N}$ representing those $\O_F$-linear isomorphisms that lie in $\Isom_{\O_F}^R$.
Define $Y^R_{\mathfrak{c},N}$ by the Cartesian diagram
\begin{equation*}
\begin{tikzcd}
Y^R_{\mathfrak{c},N} \arrow[r,hook] \arrow[d] & Y_{J_\mathfrak{c},N} \arrow[d] \\
Z^R_{\mathfrak{c},N} \arrow[r,hook]           & Z_{J_\mathfrak{c},N}.
\end{tikzcd}
\end{equation*}
Then $Y^R_{\mathfrak{c},N}$ is a fine moduli space parametrising HBAVs with a certain restriction on the interplay between polarisation and level structure.
Define the finite group $G^R_{U(N)} := R(\Z)_{>0} / R(\Z)_{>0}\cap ((1+N\O_F)^\times)^2$.
Then 
\[
\coprod_{\mathfrak{c}\in\Cl_R^+} Y_{\mathfrak{c},N}^R/G^R_{U(N)}
\]
is a model of $\Sh_{U(N)}(G,X)$ over $\Q$ that is also a coarse moduli space of HBAVs.
Looking at complex points again, one recovers the description of the set $\mathcal{A}_{U(N)}(\C)$ defined above (with $\eta U(N)$ being part of datum instead of $\eta \overline{Z(\Q)}$) by working in the category of abelian varieties ``up to isogeny'' instead.\footnote{It is curious that the adelic level structure somehow determines the abelian variety and its polarisation within its isogeny class. See \cite[\S 9]{buzz-moduli} for details. In what follows we always work with abelian varieties up to isogeny.}.


We will continue to work with the projective limit over all $U$.
For general $R$, as $\Sh(G,X)$ is a projective system of coarse moduli spaces, we can describe its $\overline{\Q}$-points $\Sh(G,X)(\overline{\Q})$ as the set $\mathcal{A}(\overline{\Q})$ of isomorphism classes of quadruples $(A,i,R(\Q)s,\eta \overline{Z(\Q)})$ as above, but with $A$ defined over $\overline{\Q}$.
The Galois group $\Gamma_\Q$ acts on an element of $\mathcal{A}(\overline{\Q})$ in the obvious way by conjugating the abelian variety and its extra structure.

On the special points of $\Sh(G,X)$, we can describe this action in more concrete terms. 
Here a point $[h,g]$ of $\Sh(G,X)(\C)$ is called \emph{special} if the Mumford--Tate group of $h$ is a torus.
By \cite[14.11]{misv}, this is the case if and only if the abelian variety $A$ of the associated quadruple $[A,i,R(\Q)s,\eta\overline{Z(\Q)}]$ under the bijection in Theorem \ref{thm:moduliHMV-general} has \emph{complex multiplication}. 
In particular, this implies that $A$ is defined over $\overline{\Q}$, hence the above quadruple defines a $\overline{\Q}$-point of $\Sh(G,X)$.

More precisely, the abelian variety $A$ having CM means that $i\colon F\to \End(A)\otimes_\Z\Q$ extends to a homomorphism $K \to \End(A)\otimes_\Z\Q$, where $K$ is a totally imaginary quadratic extension of $F$.
The field $K$ is a CM field and $\Gal(K/F)=\langle c \rangle$, where $c$ denotes complex conjugation (under any embedding of $K$ into $\C$). 
If $A$ has CM by $K$, then the pairing $s$ in the quadruple $[A,i,R(\Q)s,\eta\overline{Z(\Q)}]$ is automatically not only $F$-, but also $K$-compatible, i.\,e.\ we have
\[
s(i(k)u,v) = s(u,i(c(k))v), \quad \text{for all }u,v\in H_1(A,\Q), k\in K.
\]

Before proceeding, let us describe polarised CM abelian varieties $(A,i,s)$ more concretely as in \cite[Prop.\ 1.3]{mifundthm}.
Here $A$ denotes a complex abelian variety equipped with CM by $i\colon K \to \End(A)\otimes_\Z\Q$ and a $K$-compatible polarisation $s$.
Diagonalising the action of $K$ on the tangent space of $A$ at the origin yields an isomorphism of this tangent space with $\C^\Phi$, where $\Phi$ is a \emph{CM type} of $K$, i.\,e.\ $\Hom(K,\C)=\Phi \sqcup \overline{\Phi}$.
In other words, $\Phi$ contains precisely one embedding of every complex conjugate pair of embeddings of $K$ into $\C$.
Here $\C^\Phi:=\bigoplus_{\varphi\in\Phi} \C_\varphi$, where $\C_\varphi$ is a one-dimensional $\C$-vector space on which $K$ acts via $\varphi\colon K \hookrightarrow \C$.

We can then find an isomorphism $\xi\colon\C^\Phi/\Phi(\mathfrak{a}) \xrightarrow{~\sim~} A(\C)$ of complex Lie groups, where $\mathfrak{a}$ is a lattice in $K$ and (abusing notation) we also write $\Phi$ for the map $\Phi\colon K \to \C^\Phi$ sending $k$ to $(\varphi(k))_{\varphi\in\Phi}$.
Finally, the group $H_1(\C^\Phi/\Phi(\mathfrak{a}),\Q)$ is canonically isomorphic to $K$, so by \cite[(5.5.13)]{shi-aritheo} there exists a unique totally imaginary element $t\in K^\times$ satisfying $\im \varphi(t)>0$ for all $\varphi\in\Phi$ such that the pullback of $s$ via $\xi$ is equal to $E_t(u,v) := \Tr_{K/\Q}(tuc(v))$ for $u,v\in K$.

We call $(K,\Phi;\mathfrak{a},t)$ the \emph{type} of $(A,i,s)$.
It is determined up to changing $(\mathfrak{a},t)$ to $(\lambda\mathfrak{a},\frac{t}{\lambda c(\lambda)})$ with $\lambda\in K^\times$.
Thus an arbitrary CM point $[A,i,R(\Q)s,\eta\overline{Z(\Q)}]$ of $\Sh(G,X)$ can be written, via the isomorphism $\xi$ above, as
\begin{align}\label{eq:compCMpt}
[\C^\Phi/\Phi(\mathfrak{a}),i_\Phi,R(\Q)E_t,\eta'\overline{Z(\Q)}]
\end{align}
with $(K,\Phi;\mathfrak{a},t)$ as above, the endomorphism $i_\Phi(k)$, for $k\in K$, given as the reduction modulo $\Phi(\mathfrak{a})$ of multiplication by $\Phi(k)$ on $\C^\Phi$, and $\eta'=\xi^{-1}\circ\eta$. 

The key to describe the Galois conjugate of a CM point is \emph{Tate's half transfer}.
A published reference for the following ``Main Theorem of Complex Multiplication'' is \cite[Ch.\ 7, Thm 3.1]{la-cm}, but we usually follow the notation of \cite[\S 4]{mifundthm}. 
For a CM field $K\subset\overline{\Q}$ and a CM type $\Phi$ of $K$, Tate's half transfer is the map $F_\Phi\colon \Gamma_\Q \to \Gamma_K^{\ab}$ defined as follows:
fix coset representatives $w_\rho$ for the right $\Gamma_K$-cosets in $\Gamma_\Q$ satisfying $w_{c\rho}=c w_\rho$.
For $\gamma\in\Gamma_\Q$, define
\begin{align}\label{eq:Tate-halftr}
F_\Phi(\gamma) := \prod_{\varphi\in\Phi} \left.\left(w_{\gamma\varphi}^{-1}\gamma w_\varphi\right)\right|_{K^{\ab}}.
\end{align}
This definition is independent of the choice of coset representatives.
There is a natural lift of $F_\Phi$ under the Artin map called the \emph{Taniyama element} $f_\Phi\colon\Gamma_\Q \to \A_{K,f}^\times/K^\times$, constructed as follows.
Look at the commutative diagram with exact rows
\begin{equation}\label{diag:tate-rec}
\begin{tikzcd}
0 \arrow[r] & \ker(\rec_K) \arrow[r] \arrow[d, "1+c"] & {\A_{K,f}^\times/K^\times} \arrow[r, "\rec_K"] \arrow[d, "1+c"] & \Gamma_K^{\ab} \arrow[r] \arrow[d, "1+c"] & 0 \\
0 \arrow[r] & \ker(\rec_K) \arrow[r]                  & {\A_{K,f}^\times/K^\times} \arrow[r, "\rec_K"]                         & \Gamma_K^{\ab} \arrow[r]                  & 0.
\end{tikzcd}
\end{equation}
By \cite[Lemma 1]{tate}, $\ker(\rec_K)$ is uniquely divisible\footnote{This property also follows from the isomorphism $\ker(\rec_K)\cong \mathcal{O}_K^\times \otimes_\Z (\A_{\Q,f}/\Q)$ recalled in Section 3.2.} and complex conjugation $c$ acts trivially on it, so the left vertical arrow is an isomorphism.
By an easy diagram chase, this means that the right hand square is Cartesian.

Using the properties of the half transfer, one sees that ${}^{1+c}\! F_\Phi(\gamma) =V_{K/\Q}(\gamma)=\rec_K(\chi_{\cyc}(\gamma))$, where $V_{K/\Q}\colon \Gamma_\Q^{\ab} \to \Gamma_K^{\ab}$ denotes the transfer map (hence the name \emph{half} transfer) and $\chi_{\cyc}\colon \Gamma_{\Q} \to \hat{\Z}^\times$ denotes the cyclotomic character.
Since the right hand square in \eqref{diag:tate-rec} is Cartesian, there exists a unique $f_\Phi(\gamma)\in\A_{K,f}^\times/K^\times$ such that $\rec_K(f_\Phi(\gamma))=F_\Phi(\gamma)$ and ${}^{1+c}\!f_\Phi(\gamma)=\chi_{\cyc}(\gamma) K^\times$.

The Main Theorem of Complex Multiplication \cite[Thm 4.1]{mifundthm} now directly implies the following description of the Galois conjugate of a CM point of $\Sh(G,X)$ in terms of the Taniyama element.

\begin{thm}\label{thm:conjCMpt}
	As in \eqref{eq:compCMpt}, let $\mathcal{P}=[\C^\Phi/\Phi(\mathfrak{a}),i_\Phi,R(\Q)E_t,\eta\overline{Z(\Q)}]$ be a CM point of $\Sh(G,X)$.
	Let $\gamma\in\Gamma_\Q$ and take $f\in \A_{K,f}^\times$ such that $f_\Phi(\gamma)=fK^\times$.
	Let $\chi=\frac{\chi_{\cyc}(\gamma)}{f\cdot\overline{f}}\in F^\times$.
	Then the $\gamma$-conjugate of $\mathcal{P}$ is equal to
	\[
	[\C^{\gamma\Phi}/\gamma\Phi(f\mathfrak{a}),i_{\gamma\Phi},R(\Q) E_{\chi t},f \circ\eta\overline{Z(\Q)}].
	\]
\end{thm}

\begin{remark}
	For a lattice $\mathfrak{a}$ in $K$ and an idele $f\in\A_{K,f}^\times$, the lattice $f\mathfrak{a} \subset K$ is defined in \cite[Ch.\ 3.6, p.\ 77--78]{la-cm}.
	For example, if $\mathfrak{a}=\prod_{\mathfrak{p}} \mathfrak{p}^{a_{\mathfrak{p}}}$ is a fractional ideal of $K$, then $f\mathfrak{a}$ is given by the fractional ideal $\prod_{\mathfrak{p}} \mathfrak{p}^{a_{\mathfrak{p}}+v_{\mathfrak{p}}(f_{\mathfrak{p}})}$.
\end{remark}

\section{Plectic Galois group and action on CM points}\label{se:plecCM}

\subsection{Plectic Galois group and half transfer}\label{sse:plecgp}

\begin{defn}\label{def:plecgp}
	The \emph{plectic Galois group} is the group
	\[
	\plecgp := \plecgphash := \Aut_{\Set\text{-}\Gamma_F}(\Gamma_\Q)
	\]
	of right-$\Gamma_F$-equivariant bijections of $\Gamma_\Q$, i.\,e.\ of all bijections $\alpha\colon \Gamma_\Q \xrightarrow{~\sim~} \Gamma_\Q$ such that
	\[
	\alpha(\gamma\delta)=\alpha(\gamma)\delta \quad \text{for all } \gamma\in\Gamma_\Q, \delta\in\Gamma_F.
	\]
	The group $\plecgp$ depends on $F$; in the interest of readibility, the notation does not reflect this dependence.
	The absolute Galois group $\Gamma_\Q$ of $\Q$ embeds into $\plecgp$ by mapping $\gamma\in\Gamma_\Q$ to the map $[\gamma'\mapsto \gamma\gamma']\in\plecgp$ given by left translation by $\gamma$.
\end{defn}

There are several equivalent definitions of the plectic group.
Definition \ref{def:plecgp} is the one in \cite[\S 3]{nescholl}.
In the next two remarks we encounter two more versions of the plectic group.
For this we fix coset representatives $s_x$ for the right $\Gamma_F$-cosets in $\Gamma_\Q$, where $x\in\Gamma_\Q/\Gamma_F = \Sigma$.

\begin{remark}\label{rem:plecgpsemi}
	Let $S_\Sigma$ denote the symmetric group on the finite set $\Sigma$.
	Let $\Gamma_F^\Sigma$ denote the group of $\Sigma$-tuples $h=(h_x)_{x\in\Sigma}$ of elements of $\Gamma_F$, with the group structure given by pointwise composition.
	The group $S_\Sigma$ acts on $\Gamma_F^{\Sigma}$ by permuting the coordinates, and we can form the semi-direct product \[\plecgpsemi.\]
	
	Explicitly, the group operation is given by $(\pi,h)(\pi',h') := \left(\pi\pi',(h_{\pi'(x)} h'_x)_{x\in \Sigma}\right)$.
	The chosen coset representatives $s_x$ give a decomposition $\Gamma_\Q = \bigsqcup_{x\in\Sigma} s_x \Gamma_F$.
	By definition, $\alpha\in\plecgphash$ respects this decomposition, so for every $x\in\Sigma$ there exists $\pi(x)\in\Sigma$ such that $\alpha(s_x \Gamma_F)=s_{\pi(x)} \Gamma_F$.
	Moreover, the element $h_x=s_{\pi(x)}^{-1} \alpha(s_x)$ lies in $\Gamma_F$, and in this way we get an isomorphism $\rho_s\colon \plecgphash \xrightarrow{~\sim~} \plecgpsemi, \alpha \longmapsto (\pi,h)$.
	The isomorphism $\rho_s$ depends on the choice of coset representatives $s=(s_x)_{x\in\Sigma}$.
	However, the permutation $\pi$ of $\Sigma$ induced by $\alpha$ is independent of the choice of $s$ and will usually be denoted by $x\mapsto\alpha(x)$.
\end{remark}

\begin{remark}\label{rem:plecgpaut}
	In \cite{ne}, the plectic group is defined to be the group 
	\[
	\plecgpaut
	\]
	of $F$-algebra automorphisms of $F\otimes_\Q\overline{\Q}$.
	By \cite[(1.1.4)(iii)]{ne} or \cite[start of \S 2]{blake}, the choice of coset representatives $s=(s_x)_{x\in\Sigma}$ induces an isomorphism
	%
	$
	\beta_s\colon \plecgpaut \xrightarrow{~\sim~} \plecgpsemi.
	$
	The isomorphism $\beta_s$ depends on the choice of $s$.
	However, by \cite[Thm 2.5]{blake}, the composition $\rho_s^{-1}\circ \beta_s$ is independent of the choice of $s$ and hence gives a canonical isomorphism $\plecgpaut \xrightarrow{~\sim~} \plecgphash$.
\end{remark}

For more details on the plectic group, e.\,g.\ functoriality properties or details on the dependencies on $s$, we refer to \cite{ne,blake}.
For example, we will need the following lemma.

\begin{lem}\label{lem:prod-map}
	Let $(1,\product)$ denote the map
	$
	(1,\product)\colon \plecgpsemi \longrightarrow \Gamma_F^{\ab}
	$
	given by sending $(\pi,h)$ to $(1,\product)(\pi,h) :=\prod_{x\in\Sigma} h_x|_{F^{\ab}}.$
	Then the composition
	\[
	\Product:=(1,\product)\circ\rho_s\colon \plecgphash \longrightarrow \Gamma_F^{\ab}
	\]
	is independent of the choice of $s$.
	Moreover, restricted to the subgroup $\Gamma_\Q$ it is equal to the transfer map $V_{F/\Q}\colon \Gamma_\Q \to \Gamma_F^{\ab}$.
\end{lem}

\begin{proof}
	Let $s'\colon\Sigma\to\Gamma_\Q$ be another section, so $s'_x = s_x t_x$ for some $t=(t_x)_{x\in\Sigma}\in\Gamma_F^{\Sigma}$.
	By \cite[p.\ 7]{blake}, we have $\rho_{s'}(\alpha) = (1,t)^{-1} \rho_s(\alpha) (1,t)$ for all $\alpha\in\plecgphash$, thus $(1,\product)(\rho_{s'}(\alpha))=(1,\product)(\rho_{s}(\alpha))$.
	
	Moreover, if $\alpha$ is given by left translation by $\gamma\in\Gamma_\Q$, then $\rho_s(\alpha)$ is equal to $(\pi,(s_{\gamma x}^{-1}\gamma s_x)_{x\in\Sigma})$ and hence $P(\alpha)=V_{F/\Q}(\gamma)$ by the definition of the transfer map $V_{F/\Q}$.
\end{proof}

We are aiming for an action of (a subgroup of) $\plecgp$ on the set of CM points of $\Sh(G,X)$, extending the action of $\Gamma_\Q$.
In view of Theorem \ref{thm:conjCMpt}, for this to work we need an action of $\plecgp$ on CM types and a plectic Taniyama element $f_\Phi\colon\plecgp \to \A_{K,f}^\times/K^\times$.
The former is easy to construct; the latter will be an appropriate lift of a plectic half transfer $F_\Phi\colon\plecgp \to \Gamma_K^{\ab}$.

\begin{remark}(Plectic action on CM types)\label{rem:plec-act-CMtypes}
	Let $K\subset\overline{\Q}$ be a totally imaginary quadratic extension of $F$.
	Then $\Gamma_K$ is a subgroup of $\Gamma_F$, so any $\alpha\in\plecgphash$ induces a permutation of $\Sigma_K:=\Hom(K,\overline{\Q})=\Gamma_\Q/\Gamma_K$, which we also denote by $\alpha$.
	In this way, given a CM type $\Phi\subset\Sigma_K$ of $K$, we define $\alpha\Phi:=\{\alpha(\varphi)~|~ \varphi\in\Phi\}$.
	
	More explicitly, let $s=(s_x)_{x\in\Sigma}$ be as above, and let $\varphi_x:=s_x|_K\in\Sigma_K$.
	Then $\Sigma_K = \{c^b \varphi_x ~|~ x\in\Sigma, b\in\Z/2\Z\}$.
	Let $\alpha\in\plecgphash$ and let $\rho_s(\alpha)=(\pi,h)\in\plecgpsemi$.
	For $h'\in\Gamma_F$, define $\overline{h}'\in\Z/2\Z$ by $h'|_K = c^{\overline{h}'}\in\Gal(K/F)=\langle c\rangle$.
	Then the action of $\alpha$ on $\Sigma_K$ is given by 
	\begin{align}\label{eq:plec-act-CMtypes}
	\alpha(c^b\varphi_x) = c^{b+\overline{h}_x}\varphi_{\pi(x)},\quad x\in\Sigma,b\in\Z/2\Z.
	\end{align}
\end{remark}

%

\begin{example}
	Using \eqref{eq:plec-act-CMtypes} it is straightforward to see that the plectic group acts transitively on the set of CM types of a given CM field $K$, usually in contrast to the action of $\Gamma_\Q$.
	To give an explicit example, consider a CM field $K$ of degree $6$ over $\Q$ with cyclic Galois group $G=\Gal(K/\Q)=\langle g \rangle$.
	It is not hard to see that the set of CM types decomposes into two $G$-orbits, namely of $\Psi=\{g^0,g^1,g^2\}$ and $\Phi=\{g^0,g^4,g^2\}$ of orders $6$ and $2$ respectively.
\end{example}

Now let $K$ be a CM field whose maximal totally real subfield is equal to $F$, and let $\Phi$ be a CM type of $K$.
By \cite[p.\ 8]{blake}, the definition of Tate's half transfer $F_\Phi\colon\Gamma_\Q\to\Gamma_K^{\ab}$ in \eqref{eq:Tate-halftr} extends to a \emph{plectic half transfer} $F_\Phi\colon \plecgp \to \Gamma_K^{\ab}$ as follows.
Again fix coset representatives $w_\rho$ for $\rho\in\Sigma_K=\Gamma_\Q/\Gamma_K$ satisfying $w_{c\rho} = c w_\rho$.
Let $\alpha\in\plecgp=\plecgphash$.
In Remark \ref{rem:plec-act-CMtypes} we saw that $\alpha$ induces a permutation of $\Sigma_K$, so for $\rho\in\Sigma_K$ the element $w^{-1}_{\alpha(\rho)} \alpha(w_\rho)$ lies in $\Gamma_K$ and we define
\begin{align}\label{eq:plec-halftr}
F_\Phi(\alpha) := \prod_{\varphi\in\Phi} \left.\left(w^{-1}_{\alpha(\varphi)} \alpha(w_\varphi)\right)\right|_{K^{\ab}}.
\end{align}
By \cite[Prop.\ 3.2]{blake}, this is independent of the choice of $w_\rho$.
Also, if $\alpha$ is given by left translation by $\gamma\in\Gamma_\Q$, we recover \eqref{eq:Tate-halftr}.

\begin{remark}
	The plectic half transfer was originally defined in \cite[(2.1.3)]{ne} (resp.\ \cite[(2.1.7)]{ne}) as a map with domain $\plecgpsemi$ (resp.\ $\plecgpaut$).
	Using Remarks \ref{rem:plecgpsemi} and \ref{rem:plecgpaut}, these definitions agree with the definition given here by \cite[Thm 3.4]{blake}.
\end{remark}


\begin{remark}
	For each $x\in\Sigma$, we define the complex conjugation corresponding to $x$ to be the element $c_x\in\Gamma_F^{\ab}$ defined as follows: if $s\colon\Sigma\to\Gamma_\Q$ is a section as before, then $s_x^{-1} c s_x$ is an element of $\Gamma_F$ and its image $c_x$ in $\Gamma_F^{\ab}$ is independent of the choice of $s$.
	
	Define the subgroup $\mathfrak{c}:=\langle c_x\colon x\in \Sigma \rangle \subset \Gamma_F^{\ab}$.
	By \cite[(1.3.1)]{ne}, the Artin map $\rec_F\colon \A_{F,f}^\times/F^\times_{>0} \to \Gamma_F^{\ab}$ induces a bijection
	\begin{align}\label{eq:cft-complexconj}
	\rec_F\colon F^\times/F^\times_{>0} \xrightarrow{~\sim~} \mathfrak{c}, \quad \alpha F^\times_{>0} \longmapsto \prod_{x\in\Sigma} c_x^{\alpha_x},
	\end{align}
	where the $\alpha_x\in\Z/2\Z$ are determined by $(-1)^{\alpha_x}=\sgn(x(\alpha))$ for each $x\in\Sigma$.
	Note that the group $F^\times/F^\times_{>0}$ is isomorphic to $\{\pm 1\}^{\Sigma}$ via $\alpha F^\times_{>0} \mapsto (\sgn(x(\alpha)))_{x\in\Sigma}$.
\end{remark}


\begin{lem}\label{lem:halftr-prop}
	\begin{enumerate}
		\item\label{item:halftr-cocyc} For $\alpha, \alpha'\in\plecgp$ we have $F_\Phi(\alpha\alpha') = F_{\alpha'\Phi}(\alpha) F_\Phi(\alpha')$.
		\item\label{item:halftr-Fab} For $\alpha\in\plecgp$, let $m_x\in\Z/2\Z$ be equal to $0$ if and only if the unique elements of $\Phi$ and $\alpha\Phi$ lying above $x$ are the same. Then
		\[
		\left. F_\Phi(\alpha)\right|_{F^{\ab}}=\Product(\alpha) \prod_{x\in\Sigma} c_x^{m_x}.
		\]
	\end{enumerate}
\end{lem}

\begin{proof}
	\eqref{item:halftr-cocyc} follows from \cite[(2.1.4)(i)]{ne} and \eqref{item:halftr-Fab} follows from \cite[(2.1.4)(ii)]{ne}.
	%
\end{proof}

\subsection{Plectic Taniyama element and action on CM points}\label{sse:plecTani}

Let $(K,\Phi)$ be as before.
We want to define a \emph{plectic Taniyama element} $f_{\Phi}\colon \plecgp \to \A_{K,f}^\times/K^\times$ such that the following two properties hold:
on the one hand, for $\alpha\in\plecgp$ we want $r_K(f_\Phi(\alpha))=F_\Phi(\alpha)$, and on the other hand, if $\alpha\in\plecgp$ is given by left translation by $\gamma\in\Gamma_\Q$ we want $f_\Phi(\alpha)$ to agree with the usual Taniyama element $f_\Phi(\gamma)$ defined using the Cartesian diagram \eqref{diag:tate-rec}.
We start by looking at the commutative diagram with exact rows
\begin{equation}\label{diag:tate-extended}
\begin{tikzcd}
0 \arrow[r] & \ker(\rec_K) \arrow[r] \arrow[d, "N_{K/F}"] & {\A_{K,f}^\times/K^\times} \arrow[r, "\rec_K"] \arrow[d, "N_{K/F}"] & \Gamma_K^{\ab} \arrow[r] \arrow[d, "\res"] & 0 \\
0 \arrow[r] & \ker(\rec_F) \arrow[r] \arrow[d, "i_{K/F}"] & {\A_{F,f}^\times/F^\times_{>0}} \arrow[r, "\rec_F"] \arrow[d, "i_{K/F}"] & \Gamma_F^{\ab} \arrow[r] \arrow[d, "V_{K/F}"] & 0 \\
0 \arrow[r] & \ker(\rec_K) \arrow[r]                  & {\A_{K,f}^\times/K^\times} \arrow[r, "\rec_K"]                         & \Gamma_K^{\ab} \arrow[r]                  & 0.
\end{tikzcd}
\end{equation}

Here $N_{K/F}$ denotes the norm map, $i_{K/F}$ is induced from the inclusion $F\subset K$, $\res(\gamma)=\gamma|_{F^{\ab}}$ is the restriction and $V_{K/F}$ the transfer map.
The vertical composites are equal to $1+c$, so that forgetting the middle row yields diagram \eqref{diag:tate-rec}.
By \cite[(1.2.2)]{ne}, we have $\ker(\rec_K)\cong \O_K^\times\otimes_\Z (\A_{\Q,f}/\Q)$ and $\ker(\rec_F)\cong \O_{F,>0}^\times\otimes_\Z (\A_{\Q,f}/\Q)$.
By Dirichlet's Unit Theorem \cite[Thm I.7.4]{neuk} the groups $\O_K^\times$ and $\O_{F,>0}^\times$ have the same $\Z$-rank.
Hence the maps $N_{K/F}\colon \O_K^{\times} \to \O_{F,>0}^\times$ and $i_{K/F}\colon \O_{F,>0}^\times \to \O_K^{\times}$ have finite kernel and cokernel, and since $\A_{\Q,f}/\Q$ is a $\Q$-vector space, we conclude that both left vertical arrows $N_{K/F}\colon \ker(\rec_K) \to \ker(\rec_F)$ and $i_{K/F}\colon \ker(\rec_F) \to \ker(\rec_K)$ are in fact isomorphisms.
By the same diagram chase as in \eqref{diag:tate-rec} this means that both right hand squares in \eqref{diag:tate-extended} are Cartesian.

Let us focus on the short exact sequence in the middle, i.\,e.\
\begin{equation*}
\begin{tikzcd}
0 \arrow[r] & \ker(\rec_F) \arrow[r, "\kappa_F"] & {\A_{F,f}^\times/F^\times_{>0}} \arrow[r, "\rec_F"] & \Gamma_F^{\ab} \arrow[r] & 0.
\end{tikzcd}
\end{equation*}
As mentioned before, the kernel $\ker(\rec_F)$ is uniquely divisible, hence in particular an injective object in the category of abelian groups.
This implies the existence of a map $\omega_F\colon \A_{F,f}^\times/F^\times_{>0} \to \ker(\rec_F)$ such that $\omega_F \circ \kappa_F=\id$.
Thus the short exact sequence splits, and by the Splitting Lemma \cite[p.\ 147]{hatcher} this is equivalent to the existence of a homomorphism $\chi_F\colon \Gamma_F^{\ab} \to \A_{F,f}^\times/F^\times_{>0}$, called a \emph{splitting} of $\rec_F$, satisfying $\rec_F\circ\chi_F=\id$.
The two maps $\omega_F$ and $\chi_F$ are related by $\ker(\omega_F)=\image(\chi_F)$.

\begin{lem}\label{lem:extra-split}
	\begin{enumerate}
		\item\label{item:extra-split-cc} Restricted to the subgroup $\mathfrak{c}\subset \Gamma_F^{\ab}$, any splitting $\chi_F$ is an inverse to the isomorphism in \eqref{eq:cft-complexconj}.
		
		\item\label{item:extra-split-cycl} We can choose the splitting $\chi_F$ so that the following diagram commutes
		\begin{equation*}
		\begin{tikzcd}
		\hat{\Z}^\times \arrow[d, "i_{F/\Q}"'] & \Gamma_\Q^{\ab} \arrow[l, "\chi_{\cyc}"', "\sim"] \arrow[d, "V_{F/\Q}"] \\
		{\A_{F,f}^\times/F^\times_{>0}}        & \Gamma_F^{\ab}. \arrow[l, "\chi_F"']                            
		\end{tikzcd}
		\end{equation*}
	\end{enumerate}
\end{lem}

\begin{proof}
	For \eqref{item:extra-split-cc} let $\omega_F$ be as above.
	Since $\omega_F$ restricted to $F^\times/F^\times_{>0}$ is a group homomorphism with domain a finite group and target a uniquely divisible group, this restriction must be trivial.
	Now let $c'\in\mathfrak{c}$.
	We need to show that $\chi_F(c')=\xi(c')$, where $\xi\colon \mathfrak{c} \to F^\times/F^\times_{>0}$ denotes the inverse of $r_F$ in \eqref{eq:cft-complexconj}.
	We have $\xi(c')\in F^\times/F^\times_{>0}\subset \ker(\omega_F)=\image(\chi_F)$, thus there exists $\delta\in\Gamma_F^{\ab}$ with $\chi_F(\delta)=\xi(c')$.
	Using that $\xi$ is inverse to $r_F$, the choice of $\delta$, and the equality $r_F \chi_F= \id$, we conclude that
	\[
	\chi_F(c')=\chi_F r_F \xi(c') = \chi_F r_F \chi_F(\delta) = \chi_F(\delta) = \xi(c').
	\]

	For \eqref{item:extra-split-cycl} look at the commutative diagram
	\begin{equation}\label{diag:cft-transfer-Zhat}
	\begin{tikzcd}
	& & \hat{\Z}^\times \arrow[r, "\sim", "\rec_{\Q}"'] \arrow[d, hook, "i_{F/\Q}"] & \Gamma_\Q^{\ab} \arrow[d, hook, "V_{F/\Q}"] & \\
	0 \arrow[r] & \ker(\rec_F) \arrow[r, "\kappa_F"] & {\A_{F,f}^\times/F^\times_{>0}} \arrow[r, "\rec_F"] & \Gamma_F^{\ab} \arrow[r] & 0.
	\end{tikzcd}
	\end{equation}
	It shows that $\image(i_{F/\Q})\cap\ker(\rec_F) = 0$, hence the map \[\ker(\rec_F) \xrightarrow{\kappa_F} \A_{F,f}^\times/F^\times_{>0} \longrightarrow \A_{F,f}^\times/i_{F/\Q}(\hat{\Z}^\times) F^\times_{>0}\] is injective.
	We may therefore choose $\omega_F\colon \A_{F,f}^\times/F^\times_{>0} \to \ker(\rec_F)$ to factor through $\A_{F,f}^\times/i_{F/\Q}(\hat{\Z}^\times) F^\times_{>0}$.
	We claim that any splitting $\chi_F$ of $r_F$ with $\image(\chi_F) = \ker(\omega_F)$ (for this choice of $\omega_F$) makes the diagram in \eqref{item:extra-split-cycl} commute.
	To see this, let $\gamma\in\Gamma_{\Q}^{\ab}$ and let $z:=\chi_{\cyc}(\gamma)\in\hat{\Z}^\times$, so that $r_\Q(z)=\gamma$.
	Since $i_{F/\Q}(z)\in\ker(\omega_F)=\image(\chi_F)$, there exists an element $\delta\in\Gamma_F^{\ab}$ such that $\chi_F(\delta)=i_{F/\Q}(z)$.
	Using the definition of $z$ and $\delta$ as well as \eqref{diag:cft-transfer-Zhat} and $r_F\chi_F = \id$ yields
	\[
	\chi_F V_{F/\Q}(\gamma) = \chi_F V_{F/\Q} r_\Q(z) = \chi_F r_F i_{F/\Q}(z) = \chi_F r_F \chi_F(\delta) = \chi_F(\delta) = i_{F/\Q}(z),
	\]
	which equals $i_{F/\Q}\chi_{\cyc}(\gamma)$, i.e. \eqref{item:extra-split-cycl} commutes.
\end{proof}

From now on, fix a splitting $\chi_F$ satisfying the additional property \eqref{item:extra-split-cycl} in Lemma \ref{lem:extra-split}.

\begin{defn}\label{def:plecTanielt}
	Let $\alpha\in\plecgp$.
	By the top right Cartesian square in \eqref{diag:tate-extended}, there exists a unique element $f_\Phi(\alpha)\in\A_{K,f}^\times/K^\times$ such that $\rec_K(f_\Phi(\alpha))=F_\Phi(\alpha)$ and $N_{K/F}(f_\Phi(\alpha))=\chi_F(F_\Phi(\alpha)|_{F^{\ab}})$.
	We call the map
	\[
	f_\Phi\colon \plecgp \to \A_{K,f}^\times/K^\times
	\]
	the \emph{plectic Taniyama element}.
\end{defn}

\begin{remark}
	The definition of $f_\Phi$ depends on the choice of splitting $\chi_F$.
\end{remark}

\begin{lem}\label{lem:sameTanielt}
	Let $\gamma\in\Gamma_\Q$ and let $\alpha\in\plecgphash$ be given by left translation by $\gamma$.
	Then
	\[
	f_\Phi(\gamma) = f_\Phi(\alpha).
	\]
\end{lem}

\begin{proof}
	By Definition \eqref{eq:plec-halftr}, we have $F_\Phi(\gamma)=F_\Phi(\alpha)$.
	We also have $r_K(f_\Phi(\gamma))=r_K(f_\Phi(\alpha))=F_\Phi(\gamma)$.
	Moreover, $f_\Phi(\gamma)$ is uniquely determined by the condition ${}^{1+c}\!f_\Phi(\gamma)=\chi_{\cyc}(\gamma) K^\times$, so it is enough to show that ${}^{1+c}\!f_\Phi(\alpha)=\chi_{\cyc}(\gamma) K^\times$.
	
	Using $1+c=i_{K/F}\circ N_{K/F}$ followed by the defining property of $f_\Phi(\alpha)$, then using part \eqref{item:halftr-Fab} of Lemma \ref{lem:halftr-prop} together with the final part of Lemma \ref{lem:prod-map}, and finally applying Lemma \ref{lem:extra-split} yields 
	\begin{align*}
	{}^{1+c}\!f_\Phi(\alpha) &= i_{K/F} N_{K/F}(f_\Phi(\alpha)) \\
	&= i_{K/F}\left( \chi_F\left( F_\Phi(\alpha)|_{F^{\ab}} \right) \right) \\
	&= i_{K/F}\left( \chi_F\left( \left(\prod_{x\in\Sigma} c_x^{m_x}\right) V_{F/\Q}(\gamma) \right) \right) \\
	&= i_{K/F}\left( (m F^\times_{>0}) (i_{F/\Q}\chi_{\cyc}(\gamma) F^\times_{>0}) \right) \\
	&= \chi_{\cyc}(\gamma) K^\times,
	\end{align*}
	where $m\in F^\times$ satisfies $\sgn(x(m))=(-1)^{m_x}$.
	This finishes the proof.
\end{proof}

\begin{remark}
	In \cite[(2.2.2)]{ne} the plectic Taniyama element $\widetilde{f}_\Phi(g)$ is defined for elements $g$ of a certain subgroup $\plecgpaut_1$ of $\plecgpaut$.
	It is precisely the splitting $\chi_F$ that allows us to define a plectic Taniyama element $f_\Phi$ on the entire plectic group $\plecgp$.
	Using \cite[(2.2.3)(vi)]{ne}, the same calculation as in the proof of Lemma \ref{lem:sameTanielt} shows that $f_\Phi$ extends $\widetilde{f}_\Phi$.
\end{remark}

Let $R$ be a $\Q$-algebraic torus with associated Shimura datum $(G,X)$ as in Section \ref{sse:hmv}.
Write $R(\Q)_{>0}:= R(\Q)\cap F^\times_{>0}$ and $R(\R)_{>0}:=R(\R)\cap (F\otimes_\Q\R)^\times_{>0}$, where $(F\otimes_\Q\R)^\times_{>0}$ is the preimage of $\R^\Sigma_{>0}$ under the isomorphism $(F\otimes_\Q\R)\cong\R^\Sigma$.

\begin{defn}\label{def:plec-CM}
	We define the subgroup $\plecgpcm$ of $\plecgp$ by the Cartesian diagram
	\begin{equation*}
	\begin{tikzcd}
	\plecgpcm \arrow[d] \arrow[r, hook]                            & \plecgp \arrow[d, "{\chi_F\circ\Product}"] \\
	R(\A_f)/R(\Q)_{>0} \arrow[r, hook] & \A_{F,f}^\times/F^\times_{>0}.          
	\end{tikzcd}
	\end{equation*}
	We will also write $(\plecgphash)^R_{\CM}$ for $\plecgpcm$, and denote the subgroup of $\plecgpsemi$ (resp.\ $\plecgpaut$) isomorphic to $(\plecgphash)^R_{\CM}$ under the isomorphism of Remark \ref{rem:plecgpsemi} (resp.\ Remark \ref{rem:plecgpaut}) by $(\plecgpsemi)^R_{\CM}$ (resp.\ $\plecgpaut^R_{\CM}$).
	By Lemma \ref{lem:prod-map} this is independent of the choice of $s$ (on which these isomorphisms depend).
\end{defn}

\begin{example}
	For $R=R_{F/\Q}\G_m$ we clearly have $\plecgpcm=\plecgp$.
	For $R=\G_m$, using part \eqref{item:extra-split-cycl} of Lemma \ref{lem:extra-split} one sees that the group $\plecgpaut^{\G_m}_{\CM}$ is equal to the group $\plecgpaut_0$ of \cite[\S 2.2]{ne}, i.\,e.\ $\Gamma^{\pl,\G_m}_{\CM}$ fits into the Cartesian diagram
	\begin{equation}\label{diag:plecgpcm-Nek}
	\begin{tikzcd}
	{\Gamma^{\pl,\G_m}_{\CM}} \arrow[r, hook] \arrow[d] & \plecgp \arrow[d, "\Product"] \\
	\Gamma_{\Q}^{\ab} \arrow[r, "V_{F/\Q}", hook]       & \Gamma_F^{\ab}.               
	\end{tikzcd}
	\end{equation}
	In particular, for $R=R_{F/\Q}\G_m$ and $R=\G_m$ the group $\plecgpcm$ does not depend on the choice of $\chi_F$.
	We do not know if this is the case for arbitrary $R$.
\end{example}

\begin{remark}\label{rem:Galinpleccmgp}
	For arbitrary $R$, the canonical embedding of $\Gamma_\Q$ into $\plecgp$ factors through $\plecgpcm$ because if $\alpha\in\plecgp$ is given by left translation by $\gamma\in\Gamma_\Q$, then $\chi_F\circ P(\alpha)=\chi_{\cyc}(\gamma)F^\times_{>0}$ by the final part of Lemma \ref{lem:prod-map} and part \eqref{item:extra-split-cycl} of Lemma \ref{lem:extra-split}, and $\chi_{\cyc}(\gamma)$ lies in $\A_{f}^\times \subset R(\A_f)$. 
\end{remark}

\begin{thm}\label{thm:plecact-CM}
	Let $(K,\Phi;\mathfrak{a},t)$ be a type and let $[\C^\Phi/\Phi(\mathfrak{a}),i_\Phi,R(\Q) E_t,\eta\overline{Z(\Q)}]$ be a CM point of $\Sh(G,X)$ as in \eqref{eq:compCMpt}.
	Let $\alpha\in\plecgpcm$ and $f\in\A_{K,f}^\times$ such that $f_\Phi(\alpha) = f K^\times\in\A_{K,f}^\times/K^\times$.
	Let $u\in R(\A_f)$ be such that $\chi_F\circ \Product(\alpha) = u F^\times_{>0}$.
	Finally let $\chi:=\frac{u}{N_{K/F} f}\in\A_{F,f}^\times$.
	Then $\chi$ lies in $F^\times$.
	Furthermore, define
	\begin{align}\label{eq:def-plec-act-CMpts-general}
	\alpha\left[\C^\Phi/\Phi(\mathfrak{a}),i_\Phi,R(\Q) E_t,\eta\overline{Z(\Q)}\right] := \left[\C^{\alpha\Phi}/\alpha\Phi(f \mathfrak{a}),i_{\alpha\Phi},R(\Q) E_{\chi t},f\circ\eta\overline{Z(\Q)}\right].
	\end{align}
	
	This defines a group action of $\plecgpcm$ on the set of CM points of $\Sh(G,X)$, extending the action of $\Gamma_\Q$.
\end{thm}

\begin{proof}
	We write $\mathcal{P}$ for $[\C^\Phi/\Phi(\mathfrak{a}),i_\Phi,R(\Q) E_t,\eta\overline{Z(\Q)}]$.
	It is easy to check that \eqref{eq:def-plec-act-CMpts-general} does not depend on the choices of $f$, $u$, or the tuple $(\C^\Phi/\Phi(\mathfrak{a}),i_\Phi,R(\Q) E_t,\eta\overline{Z(\Q)})$ in the isomorphism class $\mathcal{P}$.
	We now show that the right hand side of \eqref{eq:def-plec-act-CMpts-general} does indeed define a CM point of $\Sh(G,X)$.
	
	For this we need to check two things.
	First of all, we want that $f\circ\eta$ sends $R(\A_f)\psi$ to $R(\A_f)E_{\chi t}$.
	For $v,w\in V\otimes_\Q \A_f$, we have
	\begin{align*}
	E_{\chi t}\left(f\circ\eta(v), f\circ\eta(w)\right) &= \Tr_{K/\Q}\left(u t \eta(v) \overline{\eta(w)}\right) \\
	&= E_t\left(\eta(u v), \eta(w)\right) \\
	&= \psi\left(u r v,w\right), \quad \text{for some } r\in R(\A_f).
	\end{align*}
	By definition of the group $\plecgpcm$ we have $u\in R(\A_f)$, so $ur\in R(\A_f)$ as desired.
	
	Secondly, we want that $\im \rho(\chi t)>0$ for all $\rho\in\alpha\Phi$.
	We use the definition of $f$ and $f_\Phi(\alpha)$ followed by part \eqref{item:halftr-Fab} of Lemma \ref{lem:halftr-prop} and finally the choice of $u$ together with part \eqref{item:extra-split-cc} of Lemma \ref{lem:extra-split} to calculate
	\[
	(N_{K/F} f) F^\times_{>0} = \chi_F\left(\left. F_\Phi(\alpha) \right|_{F^{\ab}}\right)
	= \chi_F\left( P(\alpha) \prod_{x\in\Sigma} c_x^{m_x} \right)
	= (u F^\times_{>0}) (m F^{\times}_{>0}),
	\]
	where the $m_x$ are defined in Lemma \ref{lem:halftr-prop} and $m\in F^\times$ satisfies $\sgn(x(m))=(-1)^{m_x}$.
	We conclude that
	\begin{align*}
	\chi\in m^{-1} F^\times_{>0},
	\end{align*}
	and since $\im \varphi(t)>0$ for all $\varphi\in\Phi$ we precisely get $\im \rho(\chi t)>0$ for all $\rho\in\alpha\Phi$.
	
	Thus $\alpha\mathcal{P}$ is well-defined.
	That it defines a group action follows easily from the cocycle relation in part \eqref{item:halftr-cocyc} of Lemma \ref{lem:halftr-prop}.
	Finally, the action extends the action of $\Gamma_\Q$ by Theorem \ref{thm:conjCMpt} and Lemma \ref{lem:sameTanielt} combined with part \eqref{item:extra-split-cycl} of Lemma \ref{lem:extra-split}.
\end{proof}

\begin{remark}
	For $R=\G_m$, Theorem \ref{thm:plecact-CM} is proved (in slightly different terms) in \cite[(2.2.5)]{ne}.
	For $R=R_{F/\Q}\G_m$, Theorem \ref{thm:plecact-CM} is stated without proof in \cite[Prop.\ 6.8]{nescholl}.
\end{remark}

\section{Plectic action on connected components}\label{se:plecpi0}

Let $R$, $G$, and $X$ be as in Section \ref{sse:hmv}.
We apply general results from the theory of Shimura varieties to describe the set of geometric connected components of the Shimura variety $\Sh(G,X)$.

\begin{lem}\label{lem:pi0SV}
	Let $U$ be a sufficiently small compact open subgroup of $G(\A_f)$.
	Then
	\begin{enumerate}
		\item\label{item:pi0SV} The set of connected components of $\Sh_U(G,X)=G(\Q)\backslash [X\times G(\A_f)/U]$ is equal to
		\[
		\pi_0(\Sh_U(G,X)) = R(\Q)\backslash \left[\VZ \times R(\A_f)/d(U) \right] =:\Sh_{d(U)}(R,\VZ),
		\]
		where $d\colon G \to R$ is the determinant map, see \eqref{diag:defG}.
		
		\item\label{item:pi0map} For $[(z_x)_{x\in\Sigma},g]\in\Sh_U(G,X)$, we have
		\[
		\pi_0([(z_x)_{x\in\Sigma},g]) = [(\sgn\im z_x)_{x\in\Sigma},d(g)]\in \Sh_{d(U)}(R,\VZ).
		\]
	\end{enumerate}
\end{lem}

\begin{proof}
	Part \eqref{item:pi0SV} follows from \cite[Thm 5.17]{misv} combined with \cite[p.\ 63]{misv}.
	These results apply because $G^{\der}=R_{F/\Q}\SL_2$ is simply connected, and yield \eqref{item:pi0SV} because $G \to G/G^{\der}$ is given by $d\colon G \to R$, see Lemma \ref{lem:Gad}, and the set $Y$ in \cite{misv} is precisely equal to $\VZ$.
	
	Similarly, part \eqref{item:pi0map} follows from \cite[p.\ 59]{misv} if one uses the connected component $X^+:=\mathfrak{h}^{\Sigma}$ of $X$.
\end{proof}

\begin{remark}
	The pair $(R,\VZ)$ is not a Shimura datum in the usual sense.
	However, the double quotient $\Sh_{d(U)}(R,\VZ)$ is a zero-dimensional Shimura variety in the sense of \cite[p.\ 62--63]{misv}.
\end{remark}

\begin{remark}
	Taking the projective limit over all $U$ in Lemma \ref{lem:pi0SV} yields a map $\pi_0\colon \Sh(G,X) \to \Sh(R,\VZ)$, which can be described by the same formula.
\end{remark}

\begin{defn}\label{def:Galact-pi0}
	We let $\Gamma_\Q$ act on $\Sh(R,\VZ)$ as follows.
	Let $\gamma\in\Gamma_\Q$ and let $q\in\A_\Q^\times$ be such that $\art_\Q(q)=\gamma|_{\Q^{\ab}}$.
	Moreover, let us denote the embedding $\G_m \subset R$ by $i$.
	Then we define
	\[
	\gamma [y,g] := [i(q)_\infty y, i(q)_f g], \quad [y,g]\in\Sh(R,\VZ),
	\]
	where $i(q)=(i(q)_\infty,i(q)_f)\in R(\R) \times R(\A_f) = R(\A)$. 
\end{defn}

\begin{lem}
	The map
	\[
	\pi_0 \colon \Sh(G,X) \to \Sh(R,\VZ)
	\]
	is $\Gamma_\Q$-equivariant.
\end{lem}

\begin{proof}
	This is \cite[(64)]{misv} because for the Shimura datum $(G,X)$ the reciprocity morphism defined in \cite[(60)]{misv} is easily seen to be equal to $i\colon \G_m \to R$.
\end{proof}

For technical reasons, we use the following slightly different description of the connected components of $\Sh(G,X)$.
To that extent, let us equip $R(\A)\subset \A_F^\times$ with the subspace topology, and let $C_R:=R(\A)/R(\Q)$.
For example, for $R=R_{F/\Q}\G_m$ the group $C_R$ is equal to the idele class group $C_F$ of $F$.

\begin{lem}\label{lem:pi0CR}
	We have
	\[
	\pi_0(\Sh(G,X)) = \pi_0(C_R).
	\]
\end{lem}

\begin{proof}
	By part \eqref{item:pi0SV} of Lemma \ref{lem:pi0SV}, we have $\pi_0(\Sh(G,X)) = \varprojlim_U \Sh_U(R,\VZ)$, where $U$ runs over the compact open subgroups of $R(\A_f)$.
	By definition, $\Sh_U(R,\VZ)=\pi_0(C_R/U')$ for $U':=(R(\Q).[1\times U])/R(\Q)$, where $1$ denotes the trivial subgroup in $R(\R)$.
	As taking $\pi_0$ commutes with taking the projective limit over $U$, we are left to show that the canonical map
	\[
	p\colon C_R \longrightarrow  \varprojlim_U C_R/U'
	\]
	is an isomorphism.
	The map $p$ is injective because $R(\Q)$ is discrete in $R(\A)$ and the intersection over all $U$ is trivial.
	It is surjective because the topological group $C_R$ is complete (because it is a subgroup of $C_F$, which is complete), and $R(\Q)$ is a discrete subgroup of $R(\R)$. 
\end{proof}

From now on, let us denote the fixed embedding $\G_m \subset R$ (resp.\ $R \subset R_{F/\Q}\G_m$) by $i$ (resp.\ $j$).

\begin{lem}\label{lem:pi0-Galact}
	If we let $\gamma\in\Gamma_\Q$ act on $\pi_0(C_R)$ as multiplication by $i\left(\art_\Q^{-1}(\gamma|_{\Q^{\ab}})\right)$, then the bijection of Lemma \ref{lem:pi0CR} is $\Gamma_\Q$-equivariant.
\end{lem}

\begin{proof}
	This follows directly from the explicit description of the bijection in Lemma \ref{lem:pi0CR} and the Galois action on $\pi_0(\Sh(G,X))$ in Definition \ref{def:Galact-pi0}.
\end{proof}

Lemma \ref{lem:pi0-Galact} motivates the following definition.

\begin{defn}\label{def:pi0-plec}
	We define the group $\plecgppi$ by the Cartesian diagram
	\begin{equation}\label{diag:def-pi0gp}
	\begin{tikzcd}
	\plecgppi \arrow[d, "\plecpiprod"'] \arrow[r]                            & \plecgp \arrow[d, "{\Product}"] \\
	\pi_0(C_R) \arrow[r, "\art_F\circ j"] & \Gamma_F^{\ab}.          
	\end{tikzcd}
	\end{equation}
	Moreover, we let $\alpha\in\plecgppi$ act on $\pi_0(C_R)$ as multiplication by $\plecpiprod(\alpha)$.
\end{defn}

\begin{remark}
	If $j\colon \pi_0(C_R) \to \pi_0(C_F)$ is injective, then so is the bottom and hence also the top horizontal arrow in \eqref{diag:def-pi0gp}.
	Thus in this case the group $\plecgppi$ canonically embeds into $\plecgp$.
	However, the injectivity of $j$ may depend on the choice of $R$, so the group $\plecgppi$ is not a subgroup of $\plecgp$ in general.
	See \cite[\S 5.3]{leon-thesis} for a more detailed discussion.
\end{remark}

\begin{example}
	For $R=R_{F/\Q}\G_m$, the map $j$ is equal to the identity and $\art_F$ is an isomorphism, thus $\plecgppi=\plecgp$.
	For $R=\G_m$, the bottom left entries of diagrams \eqref{diag:plecgpcm-Nek} and \eqref{diag:def-pi0gp} are isomorphic via $\art_\Q$, and the map $j\colon \pi_0(C_\Q) \to \pi_0(C_F)$ corresponds to the transfer map $V_{F/\Q}\colon \Gamma_\Q^{\ab} \to \Gamma_F^{\ab}$ under class field theory.
	Thus the groups $\Gamma^{\pl,\G_m}_{\pi_0}$ and $\Gamma^{\pl,\G_m}_{\CM}$ are canonically isomorphic.
\end{example}

\begin{prop}\label{lem:plecCM-inside-pi0}
	The group $\plecgpcm$ canonically embeds into $\plecgppi$.
	In particular, $\Gamma_\Q$ canonically embeds into $\plecgppi$, and restricted to $\Gamma_\Q$ the actions in Definition \ref{def:pi0-plec} and in Lemma \ref{lem:pi0-Galact} agree.
\end{prop}

\begin{proof}
	Look at the diagram 
	\begin{equation}\label{diag:CMgp-inside-pi0gp}
	\begin{tikzcd}
	\plecgpcm \arrow[r, hook] \arrow[d] & \plecgp \arrow[d, "{\chi_F\circ\Product}"] \arrow[rdd, "{\Product}", bend left=49] &                                         \\
	R(\A_f)/R(\Q)_{>0} \arrow[r, hook] \arrow[d]    & {\A_{F,f}^\times/F^\times_{>0}} \arrow[rd, "\rec_F", two heads] \arrow[d, two heads]       &  \\
	\pi_0(C_R) \arrow[r, "j"]               & \pi_0(C_F) \arrow[r, "\art_F", "\sim"']  & \Gamma_F^{\ab}.
	\end{tikzcd}
	\end{equation}
	The top square is the Cartesian square in Definition \ref{def:plec-CM}, the top right hand triangle commutes because $\rec_F \circ \chi_F = \id$, and the bottom right hand triangle commutes by the relationship between $\art_F$ and $\rec_F$. 
	Moreover, the bottom vertical arrows are induced from the inclusions $R(\A_f) \subset R(\A)$ and $\A_{F,f}^\times \subset \A_F^\times$, respectively, thus the bottom square commutes by functoriality of taking quotients and applying $\pi_0$. 
	
	Looking only at the outer edges of \eqref{diag:CMgp-inside-pi0gp}, we conclude that the embedding $\plecgpcm \subset \plecgp$ factors through $\plecgppi$ by the Cartesian diagram \eqref{diag:def-pi0gp} and we get the commutative diagram
	\begin{equation}\label{diag:CMgp-pigp-incl}
	\begin{tikzcd}
	\plecgpcm \arrow[rrd, hook, bend left] \arrow[rd, dashed, hook] \arrow[rdd, "\lambda"', bend right] &                                               &                               \\
	& \plecgppi \arrow[r] \arrow[d, "\plecpiprod"'] & \plecgp \arrow[d, "\Product"] \\
	& \pi_0(C_R) \arrow[r]                          & \Gamma_F^{\ab},               
	\end{tikzcd}
	\end{equation}
	where $\lambda$ denotes the composition of the two left hand vertical arrows of \eqref{diag:CMgp-inside-pi0gp}.
	Finally, for $\gamma\in\Gamma_\Q$ and $\alpha\in\plecgpcm$ given by left translation by $\gamma$, we have
	\[
	\plecpiprod(\alpha)=\lambda(\alpha)=i\left(\art_\Q^{-1}(\gamma|_{\Q^{\ab}})\right),
	\]
	where the second equality follows from Remark \ref{rem:Galinpleccmgp}.
	This proves the last assertion. 
\end{proof}

We are ready to prove our main result.

\begin{thm}\label{thm:pi0-plec-equiv}
	Let $\Sh(G,X)_{\CM}$ denote the set of CM points of $\Sh(G,X)$.
	Then the map
	\[
	\pi_0\colon \Sh(G,X)_{\CM} \longrightarrow \Sh(R,\VZ)
	\]
	is $\plecgpcm$-equivariant.
\end{thm}

\begin{proof}
	Let
	\[
	\mathcal{P}=\left[\C^\Phi/\Phi(\mathfrak{a}),i_\Phi,R(\Q) E_t, \eta \overline{Z(\Q)}\right]
	\]
	be a CM point of $\Sh(G,X)$, with $K$ a totally imaginary quadratic extension of $F$ and $(K,\Phi;\mathfrak{a},t)$ as in \eqref{eq:compCMpt}.
	Write the CM type $\Phi$ as $\Phi=\{\varphi_x~|~ x\in\Sigma_F\}$, where $\varphi_x|_F=x$ for all $x\in\Sigma_F$.
	We assume that $\im \varphi_x(t)>0$ for all $x\in\Sigma_F$.
	The proof proceeds in three steps.
	
	\begin{itemize}
		\item[(a)] Calculation of $\pi_0(\mathcal{P})$.
	\end{itemize}
	We start with the calculation of the point $[h,g]\in\Sh(G,X)$ corresponding to $\mathcal{P}$.
	Note that $H_1(\C^\Phi/\Phi(\mathfrak{a}),\Q)=K$ and $h_\Phi\colon\mathbb{S}\to \GL(H_1(\C^\Phi/\Phi(\mathfrak{a}),\R))$, the Hodge structure of the abelian variety $\C^\Phi/\Phi(\mathfrak{a})$, is given on real points by
	\begin{align*}
	h_\Phi\colon \C^\times &\longrightarrow \Aut_{K\otimes_\Q\R}(H_1(\C^\Phi/\Phi(\mathfrak{a}),\R)) = (K\otimes_\Q\R)^\times \subset \GL_\R(\C^\Phi) \\
	i &\longmapsto \begin{pmatrix}
	0 & -1 & & & \\
	1 & 0 & & & \\
	& & \ddots & & \\
	& & & 0 & -1 \\
	& & & 1 & 0
	\end{pmatrix}
	\end{align*}
	with respect to the $\R$-basis 
	\[
	\left\{
	\begin{pmatrix}
	1 \\ 0 \\ \vdots \\ 0
	\end{pmatrix},
	\begin{pmatrix}
	i \\ 0 \\ \vdots \\ 0
	\end{pmatrix},
	\ldots,
	\begin{pmatrix}
	0 \\ \vdots \\ 0 \\ 1
	\end{pmatrix},
	\begin{pmatrix}
	0 \\ \vdots \\ 0 \\ i
	\end{pmatrix}
	\right\}
	\]
	of $\C^\Phi=K\otimes_\Q\R$.
	Let $a \colon K \xrightarrow{~\sim~} F^2=V$ be the isomorphism of $F$-vector spaces given by $1 \mapsto \begin{pmatrix} 1 \\ 0 \end{pmatrix}$ and $\frac{1}{t} \mapsto \begin{pmatrix} 0 \\ 1 \end{pmatrix}$.
	Then $a$ satisfies the condition in \eqref{eq:doublestar} for the quadruple $\mathcal{P}$ because $E_t(u,v)=-2\psi(a(u),a(v))$ for all $u,v\in K$, and $-2\in\Q^\times \subset R(\Q)$.
	A direct calculation shows that $a\circ h_\Phi \circ a^{-1}$ corresponds to $(\beta_x i)_{x\in\Sigma}$, where $\beta_x:=\im \varphi_x\left(\frac{1}{t}\right)>0$, under the isomorphism of Lemma \ref{lem:XVZ}.
	By part \eqref{item:pi0map} of Lemma \ref{lem:pi0SV} we therefore have
	\begin{align}\label{eq:pi0-of-A}
	\pi_0(\mathcal{P})=[(1)_{x\in\Sigma},d(a\circ\eta)]\in\Sh(R,\VZ).
	\end{align}
	
	\begin{itemize}
		\item[(b)] Calculation of $\pi_0(\alpha \mathcal{P})$, for $\alpha\in\plecgpcm$.
	\end{itemize}
	Let $\alpha\in\plecgpcm$ and choose $f\in\A_{K,f}^\times$ such that $f_\Phi(\alpha)=f K^\times$.
	Let $u\in R(\A_f)$ be such that $\chi_F\circ \Product(\alpha) = u F^\times_{>0}$, and let $\chi:=\frac{u}{N_{K/F} f}\in F^\times$.
	By \eqref{eq:def-plec-act-CMpts-general} we have
	\[
	\alpha\mathcal{P}=[\C^{\alpha\Phi}/\alpha\Phi(f \mathfrak{a}),i_{\alpha\Phi},R(\Q) E_{\chi t},f\circ\eta\overline{Z(\Q)}].
	\]
	Following the same calculation as in the previous step, but with $\Phi$ replaced by $\alpha\Phi$, $t$ replaced by $\chi t$, and $a$ replaced by $a'$ with $a'(1)=\begin{pmatrix} 1 \\ 0 \end{pmatrix}$ and $a'\left(\frac{1}{\chi t}\right)=\begin{pmatrix} 0 \\ 1 \end{pmatrix}$, leads to
	\[
	\pi_0(\alpha\mathcal{P})=[(1)_{x\in\Sigma},d(a'\circ f \circ\eta)]\in\Sh(R,\VZ).
	\]
	
	Using $a'=\begin{pmatrix} 1 & \\ & \chi \end{pmatrix} \circ a$ we see that
	\[
	d(a'\circ f \circ\eta)=\chi N_{K/F}(f) ~d(a\circ\eta) = u ~d(a\circ\eta),
	\]
	so we conclude that
	\begin{align}\label{eq:pi0-of-plec-conj-A}
	\pi_0(\alpha\mathcal{P}) = [(1)_{x\in\Sigma},u ~d(a\circ\eta)].
	\end{align}
	
	\begin{itemize}
		\item[(c)] Calculation of $\alpha(\pi_0(\mathcal{P}))$.
	\end{itemize}
	By Definition \ref{def:pi0-plec}, the element $\alpha$ acts on $\pi_0(\mathcal{P})$ as multiplication by $\plecpiprod(\alpha)\in \pi_0(C_R)$.
	By \eqref{diag:CMgp-pigp-incl} we have $\plecpiprod(\alpha)=\lambda(\alpha)$, which by definition of $\lambda$ and $u$ is equal to the image of $u R(\Q)_{>0}$ inside $\pi_0(C_R)$.
	Now the identification of Lemma \ref{lem:pi0CR} together with \eqref{eq:pi0-of-A} yields
	\[
	\alpha(\pi_0(\mathcal{P}))=[(1)_{x\in\Sigma},u ~d(a\circ\eta)],
	\]
	which is equal to $\pi_0(\alpha\mathcal{P})$ by \eqref{eq:pi0-of-plec-conj-A}.
\end{proof}

\begin{remark}
	A priori, the group $\plecgpcm$ as well as its action on CM points depend on the choice of splitting $\chi_F$ of $r_F$.
	However, Theorem \ref{thm:pi0-plec-equiv} shows equivariance for the $\plecgpcm$-action for any choice of $\chi_F$.
\end{remark}

\bibliographystyle{alpha}
\bibliography{literature} 
\end{document}